\documentclass[oneside,reqno]{amsart}

\calclayout

\usepackage{amsaddr,amsrefs}
\usepackage{amsfonts,amssymb,mathtools}
\usepackage{enumitem}

\usepackage[noabbrev]{cleveref}

\crefname{equation}{}{}
\Crefname{equation}{Equation}{Equations}

\crefname{lem}{lemma}{lemmas}
\Crefname{lem}{Lemma}{Lemmas}
\crefname{assum}{assumption}{assumptions}
\Crefname{assum}{Assumption}{Assumptions}

\usepackage{tikz}
\usetikzlibrary{arrows,calc}
\usepackage{pgfplots}
\pgfplotsset{compat=1.17}
\usepackage{thmtools}
\usepackage{thm-restate}

\newcommand{\mcX}{\mathcal{X}}
\newcommand{\mcO}{\mathcal{O}}
\newcommand{\mcU}{\mathcal{U}}
\newcommand{\mcQ}{\mathcal{Q}}
\newcommand{\mcH}{\mathcal{H}}
\newcommand{\mcE}{\mathcal{E}}

\newcommand{\mcF}{\mathcal{F}}

\newcommand{\ResnI}{\mathrm{Res}_n^{(1)}(t)}
\newcommand{\ResnII}{\mathrm{Res}_n^{(2)}(t)}
\newcommand{\ResI}{\mathrm{Res}^{(1)}(t)}
\newcommand{\ResII}{\mathrm{Res}^{(2)}(t)}

\newcommand{\Z}{\mathbb{Z}}
\newcommand{\R}{\mathbb{R}}
\newcommand{\N}{\mathbb{N}}

\newtheorem{defn}{Definition}[section]
\newtheorem{prop}{Proposition}[section]
\newtheorem{lem}{Lemma}[section]
\newtheorem{theorem}{Theorem}[section]
\newtheorem{assum}{Assumption}
\newtheorem{rem}{Remark}

\begin{document}
	
\title{Long-Time approximations of small-amplitude, long-wavelength FPUT solutions}

\author{Trevor Norton$^1$ and C. Eugene Wayne$^1$}
\address{$^1$Department of Mathematics and Statistics, Boston University, Boston, MA, United States}
\email{nortontm@bu.edu, cew@bu.edu}

\date{}
\thanks{The authors' research was supported in part by the US NSF through the grant DMS-18133.}

\begin{abstract}
It is well known that the Korteweg-de Vries (KdV) equation and its generalizations serve as modulation equations for traveling wave solutions to generic Fermi-Pasta-Ulam-Tsingou (FPUT) lattices. Explicit approximation estimates and other such results have been proved in this case. However, situations in which the defocusing modified KdV (mKdV) equation is expected to be the modulation equation have been much less studied. As seen in numerical experiments, the kink solution of the mKdV seems essential in understanding the \(\beta\)-FPUT recurrence. In this paper, we derive explicit approximation results for solutions of the FPUT using the mKdV as a modulation equation. In contrast to previous work, our estimates allow for solutions to be non-localized as to allow approximate kink solutions. These results allow us to conclude meta-stability results of kink-like solutions of the FPUT. 
\end{abstract}

\maketitle

\section{Introduction}

Much work has been done in studying the traveling wave solutions of the Fermi-Pasta-Ulam-Tsingou (FPUT) lattice and analyzing their stability. Recall that the FPUT lattice is an infinite set of ODEs with nearest-neighbor interaction given by the equations
\begin{equation}\label{fput-lattice-odes}
	\ddot x_n = V'(x_{n+1} - x_n) - V'(x_n - x_{n-1}), \quad n \in \Z
\end{equation}
where \(V\) is the interaction potential between neighboring particles and \(``\dot{\hspace{1em}}"\) denotes the derivative with respect to the time \(t\in \mathbb{R}\). \Cref{fput-lattice-odes} can be rewritten in the strain variables \(u_n := x_{n+1} - x_n\) so that
\begin{equation}\label{fput-lattice-equations-strain-variables}
	\ddot{u}_n = V'(u_{n+1}) - 2 V'(u_n) + V'(u_{n-1}), \quad n \in \Z.
\end{equation}
It is well-known that the continuum limit for \cref{fput-lattice-equations-strain-variables} is given by generalized Korteweg-de Vries (KdV) equations, which implies the existence of small-amplitude, long-wavelength traveling wave solutions of the FPUT. For instance, if we take a sufficiently regular \(V(x)\) with \(V''(0) =: c^2 > 0\) and \(V'''(0) \neq 0 \), then small-amplitude, long wavelength solutions of \cref{fput-lattice-equations-strain-variables} can be approximated as
\begin{equation*}
	u_n(t) \approx \epsilon^2 f(\epsilon(n + ct), \epsilon^3 t) + \epsilon^2 g(\epsilon(n-ct), \epsilon^3t)
\end{equation*}
where \(f\) and \(g\) solve the uncoupled KdV equations
\begin{align*}
	2 \partial_2 f &= \frac{ c} {12} \partial_1^3 f+ \frac{V'''(0)}{c } f \partial_1 f, \\
	-2 \partial_2 g &= \frac{ c} {12} \partial_1^3 g+ \frac{V'''(0)}{c } g \partial_1 g.
\end{align*}
This approximation holds up to order \(\mcO(\epsilon^4)\) for time scales of order \(\mcO(\epsilon^{-3})\) and can be shown by using energy methods to control the error of the approximation \cite{schneider2000counter}. Similar results have been shown using methods from dispersive PDEs \cite{hong2021korteweg}. Additionally, solitary wave solutions of the FPUT whose profiles are close to the profile of the soliton solution of the KdV exist and are asymptotically stable in an exponentially weighted space \cites{friesecke1999solitary,friesecke2002solitary,friesecke2003solitary,friesecke2004solitary}. 

For potentials of the form
\begin{equation}\label{generic-potential}
	V(x) = \frac 1 2 x^2 + \frac 1 {p+1} x^{p+1},
\end{equation}
the continuum limit of \cref{fput-lattice-equations-strain-variables} becomes the generalized KdV equation
\begin{equation}
	\partial_2 g + \frac 1 {12} \partial_1^3 g + \partial_1( g^p) = 0
\end{equation}
and there are approximate solutions of the form
\begin{equation}
	u_n(t) \approx \epsilon^{\alpha} g(\epsilon(n-t), \epsilon^3 t)
\end{equation}
with \(\alpha = 2/(p-1)\). This approximation holds for time scales of order \(\mcO( \epsilon^{-3} |\log(\epsilon)|)\) for a global solution \(g\). This extends the approximation result beyond the normal time scale of the ansatz, i.e., the ansatz was taken on the \(\epsilon^{-3}\) time scale but holds for longer times as \(\epsilon \to 0\). This gives that the orbital stability of solutions for the generalized KdV implies the metastability of solitary wave solutions for the FPUT \cite{khan2017long}. A thorough overview of the research  into solitary waves of the FPUT can be found in \cite{vainchtein2022solitary}.

For this paper, we will assume that the potential is of the form
\begin{equation}\label{potential}
	V(x) = \frac 1 2 x^2  - \frac 1 {24} x^4.
\end{equation}
The minus sign for the quartic term (in contrast to the positive coefficient in \cref{generic-potential}) gives a formal continuum limit of the defocusing modified KdV (mKdV), which can be written as
\begin{equation*}\label{mkdv}
	\partial_t v - 6v^2 \partial_x v  + \partial_x^3 v = 0,
\end{equation*}
after rescaling. This limit can be found by doing a formal calculation with the small-amplitude, long-wavelength ansatz and equating the orders of \(\epsilon\). A notable difference between the defocusing mKdV and the KdV is that the former admits kink solutions. Numerical calculations of the FPUT with potential given in \cref{potential} were carried out, and the kink solutions of the mKdV seem essential to understand the recurrence exhibited by the FPUT \cite{pace2019beta}. As opposed to the generalized KdV equations, little work has been done in obtaining rigorous results on how the defocusing mKdV can be more generally used as a modulation equation for the FPUT with potential given by \cref{potential}. Traveling waves of \cref{fput-lattice-equations-strain-variables} where the solution has fixed limits at spatial infinity have been shown to exist using a center manifold reduction \cite{iooss2000travelling}*{Thm.~5, (d)}, and this result can be extended to show that the profile of a traveling wave solution on the lattice is well-approximated by the profile of the kink solution of the mKdV \cite{norton2023kinklike}. This kink-like solution on the lattice is in the family of heteroclinic front solutions. Such solutions have previously been identified for lattice systems with nearest-neighbor interactions in \cite{truskinovsky2005kinetics}. The existence of fronts connecting oscillatory states have been shown in \cite{schwetlick2009existence} for certain double-well potentials. Variational techniques were used to show the existence of heteroclinic traveling waves connecting asymptotic states for FPUT-type systems in \cites{herrmann2011action, herrmann2010heteroclinic}. Kink-type solutions in the FPUT were also shown in \cite{gorbushin2020supersonic} for non-smooth potentials.

This paper explores how the defocusing mKdV can be generally used as a modulation equation for traveling wave solutions of the FPUT with potential given by \cref{potential} and is based off the work done in \cite{norton2023kinklike}. One difference that this work has with other previous results is that we allow the solutions of the FPUT to have non-zero limits at infinity in anticipation for the kink-like solution. Allowing the non-zero limits creates challenges in getting the approximation results since we can no longer assume that our functions lie in nice Sobolev spaces, and so we need to introduce an appropriate Banach space for the functions.

We have two main results. The first major result shows that small-amplitude, long-wavelength solutions of \cref{fput-lattice-equations-strain-variables} can be approximated by counter-propagating solutions of a decoupled system of PDEs. In particular, the mKdV and a generalized KdV equation serve as modulation equations for such solutions on long, finite time scales. This result is notable in that behavior observed in the decoupled PDEs can thus also be observed in the FPUT. The second major result shows that the approximation of small-amplitude, long-wavelength solutions can be extended to a longer time scale when reduced to a single propagating solution of the mKdV. This implies the meta-stability of the kink-like solutions of the FPUT. 

The paper will be structured as follows. In \cref{sec:ansatz} we introduce the small-amplitude, long-wavelength ansatz and define new Banach spaces appropriate for the kink solution. The ansatz is similar to the one introduced by Schneider and Wayne in \cite{schneider2000counter}: two counter-propagating waves which satisfy de-coupled PDEs and an ``interference" term. It is shown that if the functions in the ansatz are bounded by appropriate norms, then the interference term remains uniformly bounded in time. In \cref{sec:lattice-eqns}, we rewrite \cref{fput-lattice-equations-strain-variables} into a first order system. Then using our ansatz, we can find the system of equations for the error in our approximation. We aim to show that this error remains small for our time scale. In \cref{sec:estimates} we make some necessary estimates on the residual and nonlinear terms for the error equations. In \cref{sec:long-time-approx}, we show that the ansatz holds on time scales of order \(\mcO(\epsilon^{-3})\) in a result that is analogous to \cite{schneider2000counter}. In \cref{sec:meta}, we drop the counter-propagating wave ansatz and look at a single traveling wave solution. From here, we can show that the ansatz holds for time scales \(\mcO(\epsilon^{-3}|\log(\epsilon)|)\), which allows us to comment on the meta-stability of the kink-like solution in the FPUT. The proofs of technical lemmas are provided in \cref{lemma-appendix}.

\section{Counter-Propagating Waves Ansatz}\label{sec:ansatz}

We make the assumption that solutions of \cref{fput-lattice-equations-strain-variables} can be expressed as a sum of two counter-propagating small-amplitude waves, i.e., 
\begin{equation}\label{ansatz}
	u_n(t) \approx \epsilon f(\epsilon(n+t), \epsilon^3t) + \epsilon g(\epsilon(n-ct), \epsilon^3 t) + \epsilon^3\phi(\epsilon n, \epsilon t)
\end{equation}
where we allow \(f\) to have fixed non-zero limits, \(f_\pm\), at positive and negative infinity and \(\phi\) captures the interaction effects between \(f\) and \(g\). The wave speed of \(g\) is given by
\begin{equation}\label{ansatz-wave-speed}
	c = c(\epsilon, f_+) = 1 - \frac{\epsilon^2 f_+^2}{4}.
\end{equation}
The choice of wave speed for \(f\) follows from the fact that we expect such solutions to have wave speeds near  \(V''(0) = 1\), as shown in \cites{iooss2000travelling,norton2023kinklike}. The wave speed of \(g\) results in part from the background of \(f_+\). This shifts the point about which we expand and is accounted for by a re-normalization in the speed. Plugging in the ansatz in \cref{ansatz} back into \cref{fput-lattice-equations-strain-variables} and grouping terms of the same order \(\epsilon\) together gives
\begin{equation*}
	\begin{aligned}
		&\epsilon^3 \Big(\partial_1^2 f(\cdot, \epsilon^3 t) + \partial_1^2 g(\cdot, \epsilon^3 t)\Big) \\
		&+ \epsilon^5 \Big( 2 \partial_1\partial_2 f(\cdot, \epsilon^3t) - 2 \partial_1\partial_2 g(\cdot, \epsilon^3 t) - \frac{f_+^2} 2 \partial_1^2 g + \partial_2^2 \phi(\epsilon x, \epsilon t)\Big) \\
		&+ \mathcal O(\epsilon^6)\\
		&\qquad = \quad \epsilon^3 \Big(\partial_1^2 f(\cdot, \epsilon^3 t) + \partial_1^2 g(\cdot, \epsilon^3 t)\Big) \\
		&\qquad\qquad + \epsilon^5\Big(\partial_1^2  \phi(\epsilon x, \epsilon t) \\
		&\qquad\qquad\qquad - \frac 1 6 \partial_1^2 \big[f^3(\cdot, \epsilon^3 t ) + 3 f^2(\cdot,\epsilon^3)g(\cdot, \epsilon^3t) + 3 f(\cdot, \epsilon t) g^2(\cdot,\epsilon^3 t) + g^3(\cdot, \epsilon^3 t)\big] \\
		&\qquad\qquad\qquad+ \frac 1 {12} \partial_1^4 f(\cdot, \epsilon^3 t) + \frac 1 {12} \partial_1^4 g(\cdot, \epsilon^3 t)\Big) \\
		&\qquad\qquad+ \mathcal O(\epsilon^6).
	\end{aligned}
\end{equation*}
Clearly the equation will hold up to order \(\epsilon^3\). For the order \(\epsilon^5\) terms, the equation will again hold if \(f\), \(g\), and \(\phi\) satisfy
\begin{equation}\label{f-mKdV}
	2 \partial_2 f = - \frac 1 6 \partial_1(f^3) + \frac 1 {12} \partial_1^3 f,
\end{equation}
and 
\begin{equation}\label{g-gKdV}
	-2 \partial_2 g = - \frac 1 6 \partial_1 (g^3 + 3f_+ g^2) + \frac 1 {12}\partial_1^3 g,
\end{equation}
and
\begin{equation}\label{phi-pde}
	\begin{aligned}
		\partial_2^2 \phi(\xi, \tau) &= \partial_1^2\phi(\xi, \tau)- \frac 1 6 \partial_1^2 \big[ 3(f^2(\xi+\tau,\epsilon^2\tau)-f_+^2)g(\xi-c\tau,\epsilon^2\tau) + 3 (f(\xi+\tau,\epsilon^2\tau)-f_+)g^2(\xi-c\tau,\epsilon^2\tau) \big ] \\
		\phi(\xi,0) &= \partial_1 \phi(\xi, 0) = 0.
	\end{aligned}
\end{equation}
Note that \cref{f-mKdV} is the defocusing mKdV equation and \cref{g-gKdV} is a type of generalized KdV equation. This formal calculation shows that the mKdV can serve as a modulation equation. That is, for \(\epsilon\) sufficiently small, one would expect the ansatz in \cref{ansatz} to hold for time on the order of \(\epsilon^{-3}\). We make precise this notion, but we must first make decisions for the function spaces in which the functions \(f\), \(g\), and \(\phi\) must live.

A natural choice of function space for \(g\) is a Sobolev space like \(H^k(\R)\). However, for \(f\), we want to allow the possibility of the function approaching a non-zero limit at positive and negative infinity while also having sufficient regularity. 
\begin{defn}
	For \(k\in\N\), let \(\mcX^k(\R)\) be the Banach space 
	\begin{equation*}
		\mcX^k(\R) := \{f \in L^\infty(\R) \mid f'\in H^{k-1}(\R)\}
	\end{equation*}
	with norm
	\begin{equation*}
		\| f \|_{\mcX^k(\R)} := \| f \|_{L^\infty(\R)} + \| f' \|_{H^{k-1}(\R)}.
	\end{equation*}
\end{defn}
Then \(\mcX^k\) is the set of \(L^\infty\) functions which are \(k\) times weakly differentiable and whose derivatives are in \(L^2\). That this is a Banach space follows from the Banach space isomorphism
\begin{equation*}
	\mcX^k(\R) \cong L^\infty(\R) \cap \dot H^1(\R) \cap \dot H^k(\R),
\end{equation*}
where \(\dot H^k(\R)\) denotes the homogeneous Sobolev spaces. For convenience, we let \(\mcX^0(\R)\) denote \(L^\infty(\R)\).

Note that \cref{f-mKdV} has kink solutions of the form
\begin{equation}\label{kink-solutions}
	f(X,T) = \sqrt{12 v} \tanh\left( \sqrt{12v} (X-vT)\right).
\end{equation}
In particular, setting \(v = 1/24\) we get the approximate solution on the lattice given by
\begin{equation*}
	\frac \epsilon {\sqrt 2} \tanh\left(\frac \epsilon {\sqrt 2} \left(n - \left(1-\frac{\epsilon^2} {24}\right) t\right)\right).
\end{equation*} 
Comparing the above result to the kink-like solution found in \cite{norton2023kinklike}, the approximate solution seems to agree with the kink-like solution on the lattice for long periods of time (i.e.\ it should hold formally for \(t\) of order \(\mathcal O(\epsilon^{-4})\)). The space \(\mcX^k\) allows solutions of this form and thus lets us study kink-like solutions on the lattice.

We also have the following inequalities for products of functions in \(\mcX^k\) and \(H^k\) that will be useful.

\begin{restatable}{lem}{prodruleone}
	\label{prod-rule-1-lem}
	For non-negative integers \(k\), there is a \(C>0\) such that
	\begin{equation}\label{prod_rule}
		\| fg \|_{H^k} \leq C \| f \|_{\mathcal X^k} \| g \|_{H^k}
	\end{equation}
	for any \(f\in \mcX^k(\mathbb R)\) and \(g \in H^k(\mathbb R)\).
\end{restatable}

\begin{restatable}{lem}{prodruletwo}
	\label{prod-rule-2-lem}
	For non-negative integers \(k\), there is a \(C>0\) such that
	\begin{equation*}
		\| fg \|_{\mcX^k} \leq C \| f \|_{\mcX^k} \| h \|_{\mcX^k}
	\end{equation*}
	for any \(f,g\in \mcX^k(\mathbb R)\).
\end{restatable}
See \cref{lemma-appendix} for proofs.

However, for our main result, we require that \(\phi\), the term which captures the interaction effects, remains uniformly bounded for all time. Intuitively, if \(f- f_+\) and \(g\) are localized, the inhomogeneous term in \cref{phi-pde} will quickly go to zero, and the equation governing \(\phi\) will approach the homogeneous wave equation, for which Sobolev norms remain uniformly bounded. Since the two functions are localized and counter-propagating, their product will quickly decay in time as the two wave profiles move in opposite directions as seen in \cref{fig-counter-propogating-interaction}. Thus we require that \(f\) and \(g\) quickly decay to their respective limits at infinity. This is enforced by assuming the functions belong to appropriate weighted Banach spaces.

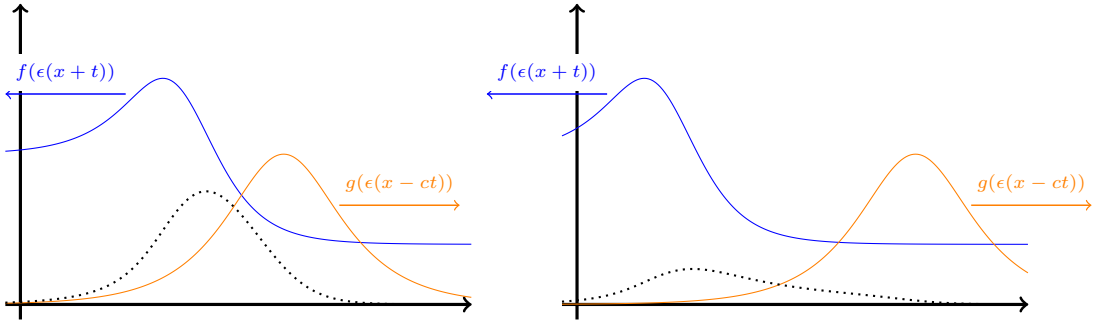
\begin{figure}[hb]
	\centering
	\begin{tikzpicture}[
		scale=2,
		declare function={sech(\x) = 2/(exp(\x) + exp(-\x));
			fun1(\x,\c) = -0.3 * tanh((\x-\c)*3) +0.7  + 0.6*sech((\x +0.25 - \c)*3.75);
			fun2(\x,\c) = sech((\x-\c)*3);
		},
		]
		\draw[->, very thick] (-0.1,0) -- (3,0) {};
		\draw[->, very thick] (0,-0.1) -- (0,2) {};
		
*		\draw[domain=-.1:3, samples = 200, color=blue] plot (\x, {fun1(\x, 1.25)} );
		\draw[domain=-.1:3, samples = 200, color=orange] plot (\x, {fun2(\x, 1.75)} );
		\draw[domain=-.1:3, samples = 200, thick, dotted] plot (\x, {2.5* (fun1(\x,1.25) - 0.4) * fun2(\x,1.75)});
		
		\draw[->,line width=0.20mm,color=blue] (.7, 1.4) -- (-0.1,1.4) {} node[midway,above=1.2, fill=white] {\scriptsize$f(\epsilon(x+t))$};
		\draw[->,line width=0.20mm,color=orange] (2.12, .66) -- (2.92,.66) {} node[midway, above=1.2] {\scriptsize $g(\epsilon(x-ct))$};
		
		\begin{scope}[xshift=3.7cm]
			\draw[->, very thick] (-0.1,0) -- (3,0) {};
			\draw[->, very thick] (0,-0.1) -- (0,2) {};
			
			\draw[domain=-.1:3, samples = 200, color=blue] plot (\x, {fun1(\x, .75)} );
			\draw[domain=-.1:3, samples = 200, color=orange] plot (\x, {fun2(\x, 2.25)} );
			\draw[domain=-.1:3, samples = 200, thick, dotted] plot (\x, {15* (fun1(\x,.75) - 0.4) * fun2(\x,2.25)});
			
			\draw[->,line width=0.20mm,color=blue] (.2, 1.4) -- (-0.6,1.4) {} node[midway,above=1.2, fill=white] {\scriptsize$f(\epsilon(x+t))$};
			\draw[->,line width=0.20mm,color=orange] (2.62, .66) -- (3.42,.66) {} node[midway, above = 1.2] {\scriptsize $g(\epsilon(x-ct))$};
		\end{scope}
	\end{tikzpicture}
	\caption{\label{fig-counter-propogating-interaction} The function \(f(\epsilon(x+t))\) (shown in blue) moves to the left while \(g(\epsilon(x-ct))\) (shown in orange) moves to the right. The product of \(f - f_+\) and \(g\) (shown by the dotted line) quickly decays in time since the functions are localized. Also note that the limit \(f_{+}\) changes the background of \(g\) which results in the change in wave speed of \(g\) as given by \cref{ansatz-wave-speed}.}
\end{figure}

A suitable choice of space for \(g\) is the weighted Sobolev spaces \(H^k_n(\R)\). Here, \(H^k_n\) for \(k,n\in\mathbb N\cup \{0\}\)
\begin{equation*}
	H^k_n(\R) := \{ g\in H^k(\R) \mid   g\langle \cdot \rangle^n \in H^k \}
\end{equation*}
where \(\langle x \rangle = \sqrt{1+x^2}\). The norm on this space is
\begin{equation*}
	\| g \|_{H^k_n(\R)} := \|  g \langle \cdot \rangle^n\|_{H^k(\R)}.
\end{equation*}
This space has the useful property that if \(g \in H^k_n\), then its Fourier transform, \(\hat g \), is in \(H^n_k\) and 
\begin{equation*}
	c \| \hat g \|_{H^n_k} \leq \| g \|_{H^k_n} \leq C \| \hat g \|_{H^n_k}
\end{equation*}
for \(c,C>0\) independent of \(g\).

We want an analogous space for \(f\), but allowing for non-zero limits at infinity. Let \(\langle\cdot \rangle_+ :\R \to \R\) be a smooth function such that
\begin{equation*}
	\langle x \rangle_+ = \begin{cases} \langle x \rangle, & x>1 \\ 1, & x<0\end{cases}
\end{equation*}
and \(\langle \cdot \rangle_+\) continued smoothly between \(0\) and \(1\) such that it is always greater than or equal to \(1\). Thus \(\langle \cdot \rangle_+\) is a function that only acts like \(\langle \cdot \rangle\) for numbers greater than \(1\). The function \(\langle \cdot \rangle_-\) is similarly defined but for numbers less than \(-1\).

\begin{defn}
	Define \(\mcX^k_{n^+} (\R)\) to be the Banach space of functions where 
	\begin{equation*}
		\mcX^k_{n^+} (\R) := \{ f \in \mcX^k(\R) \mid \lim_{x\to\infty} f(x) = f_+\text{ and } (f-f_+)\langle\cdot\rangle_+^n \in \mcX^k(\R)\}
	\end{equation*}
	with norm given by
	\begin{equation*}
		\| f \|_{\mcX^k_{n^+}(\R)} := |f_+| + \|(f-f_+) \langle \cdot \rangle_+^n \|_{\mathcal X^k(\R)}
	\end{equation*}
	Similarly, 
	\begin{equation*}
		\mcX^k_{n^-} (\R) := \{ f \in \mcX^k(\R) \mid \lim_{x\to-\infty} f(x) = f_-\text{ and } (f-f_-)\langle\cdot\rangle_-^n \in \mcX^k(\R)\}
	\end{equation*}
	and 
	\begin{equation*}
		\| f \|_{\mcX^k_{n^-}(\R)} := |f_-| + \|(f-f_-) \langle \cdot \rangle_-^n \|_{\mathcal X^k(\R)}
	\end{equation*}
	Define \(\mcX^k_n(\R)\) to be the intersection of these Banach spaces. That is,
	\begin{equation*}
		\mcX^k_n(\R) := \mcX^k_{n^+} (\R) \cap \mcX^k_{n^-} (\R), \quad \| f \|_{\mcX^k_{n} (\R)} := \|f\|_{\mcX^k_{n^+} (\R)} + \|f\|_{\mcX^k_{n^-} (\R)}.
	\end{equation*}
\end{defn}
That \(\mcX^k_{n^\pm}\) are Banach spaces follows from the fact that there exists a linear isomorphism between the Banach space \(\R\times \mcX^k\) and these spaces, which is given by
\begin{equation*}
	(\alpha, f) \mapsto \alpha + f \langle \cdot \rangle^{-n}_{\pm}.
\end{equation*}
One can show that the kink solutions as specified in \cref{kink-solutions} lie in \(\mcX^k_n\) for all \(k,n\geq 0\); the derivatives are smooth and decay exponentially to zero, and the kink solutions approach the limits \(\pm\sqrt{12v}\) exponentially fast. These spaces also contain bounded rational functions. For instance, the function \[1 + \frac 1 {x^2 +1}\] is in \(\mcX^k_2(\R)\) since it approaches its limit at infinity (which in this case is \(1\)) at a rate of \(\mathcal O(1/x^2)\), and its derivatives are in \(H^0_2(\R)\).

The definitions above are used to prove that \(\phi\) remains bounded for all time. The idea behind the proof is similar to that of \cite{schneider2000counter}*{Lemma~3.1}. The following lemma will be useful in showing the decay in products of \(f-f_+\) and \(g\).
\begin{restatable}{lem}{cknormbound}
	For each \(k\geq 0\) and \(c > 0\), there exists \(C> 0\) depending only on \(k\) such that 
	\begin{equation}\label{Ck-bound}
		\left \| \frac 1 {\langle \cdot +\tau\rangle_+^2 \langle \cdot - c\tau \rangle^2} \right \|_{C^k} \leq C\, \sup_{x\in\mathbb R} \frac 1 {\langle x +\tau\rangle_+^2 \langle x -c \tau \rangle^2}.
	\end{equation}
	Furthermore,
	\begin{equation}\label{sup-integrable}
		\int_0^\infty \sup_{x\in\mathbb R} \frac 1 {\langle x +\tau\rangle_+^2 \langle x -c \tau \rangle^2}\, d\tau <\infty.
	\end{equation}
\end{restatable}
See \cref{lemma-appendix} for proof. 

We are now ready to prove that \(\phi\) (and its time derivative) remain uniformly bounded in time.
\begin{prop}
	Fix \(T_0> 0\) and suppose that \(f \in C([-T_0,T_0], \mcX_2^{k+1}(\R))\) and \(g \in C([-T_0,T_0], H^{k+1}_2(\R)) \), with \(k\geq 3\) an integer. Also, suppose that \(f(X,T)\to f_+\) as \(X\to \infty\) for any \(T\in[-T_0,T_0]\). Then there exists a constant \(C>0\) such that 
	\begin{equation}\label{phi-bound}
		\sup_{t\in[-\epsilon^{-3}T_0,\epsilon^{-3}T_0]} \|\phi(\cdot,\epsilon t)\|_{H^k} \leq C \Bigg( \sup_{t\in[-\epsilon^{-3}T_0,\epsilon^{-3}T_0]} \left\{\| f(\cdot, \epsilon^3t) \|_{\mathcal X^{k+1}_{2}},  \| g(\cdot, \epsilon^3t) \|_{H^{k+1}_2} \right\}\Bigg)^3
	\end{equation}
	and
	\begin{equation}\label{psi-bound}
		\sup_{t\in[-\epsilon^{-3}T_0,\epsilon^{-3}T_0]} \|\psi(\cdot,\epsilon t)\|_{H^{k-1}} \leq C \Bigg( \sup_{t\in[-\epsilon^{-3}T_0,\epsilon^{-3}T_0]} \left\{\| f(\cdot, \epsilon^3t) \|_{\mathcal X^{k+1}_{2}},  \| g(\cdot, \epsilon^3t) \|_{H^{k+1}_2} \right\}\Bigg)^3,
	\end{equation}
	where \(\psi = \partial_2 \phi\).
\end{prop}

\begin{proof}
	Set \(\partial_2 \phi = \psi\). Taking the Fourier transform \(\mathcal F\) on both sides of \cref{phi-pde} and writing the ODE as a first order system, we get that 
	\begin{equation*}
		\begin{aligned}
			&\partial_2 \begin{bmatrix} \hat \phi(k,\tau) \\ \hat \psi(k,\tau) \end{bmatrix} = \begin{bmatrix}\hat \psi(k,\tau) \\ -k^2 \hat\phi(k,\tau) \end{bmatrix} \\ &+ \begin{bmatrix}
				0 \\   \frac 1 2 k^2 \mathcal F[ (f^2(\cdot+\tau),\epsilon^2\tau)-f_+^2)g(\cdot-c\tau,\epsilon^2\tau) +(f(\cdot+\tau,\epsilon^2\tau)-f_+)g^2(\cdot-c\tau,\epsilon^2\tau)](k)
			\end{bmatrix}.
		\end{aligned}
	\end{equation*}
	The semigroup generated by the linear part can be computed explicitly to be
	\begin{equation*}
		\begin{bmatrix}
			\cos(k\tau)  & \frac 1 k \sin(k \tau) \\
			-k\sin(k\tau) & \cos(k\tau)
		\end{bmatrix}.
	\end{equation*} 
	Putting the solution into variation of constants form with initial conditions set to zero gives
	\begin{equation*}
		\begin{aligned}
			&\hat  \phi(k,T) = \frac 1 2 \int_0^Tk\sin(k(T-\tau)) \times\\
			&\quad\mathcal F[ (f^2(\cdot+\tau),\epsilon^2\tau)-f_+^2)g(\cdot-c\tau,\epsilon^2\tau) +(f(\cdot+\tau,\epsilon^2\tau)-f_+)g^2(\cdot-c\tau,\epsilon^2\tau)](k)\, d\tau
		\end{aligned}
	\end{equation*}
	and 
	\begin{equation}\label{psi-fourier-transform}
		\begin{aligned}
			&\hat  \psi(k,T) = \frac 1 2 \int_0^Tk^2\cos(k(T-\tau)) \times\\
			&\quad\mathcal F[ (f^2(\cdot+\tau,\epsilon^2\tau)-f_+^2)g(\cdot-c\tau,\epsilon^2\tau) +(f(\cdot+\tau,\epsilon^2\tau)-f_+)g^2(\cdot-c\tau,\epsilon^2\tau)](k)\, d\tau
		\end{aligned}
	\end{equation}
	Hence we can get that 
	\begin{equation}\label{phi-sobolev-bound}
		\begin{aligned}
			&\|\phi(\cdot, T) \|_{H^k} \\
			&\quad\leq C \| \hat\phi(\cdot, T) \|_{H^0_k} \\
			&\quad\leq C \int_0^T \| \partial_1 ((f^2(\cdot+\tau)-f_+^2)g(\cdot-c\tau)) \|_{H^k} + \| \partial_1( (f(\cdot+\tau)-f_+)g^2(\cdot-c\tau)) \|_{H^k} \, \mathrm d \tau \\
			&\quad\leq C \int_0^T \| f(\cdot+\tau)\partial_1 f(\cdot + \tau)g(\cdot-c\tau) \|_{H^k} + \|(f^2(\cdot+\tau) -f_+^2) \partial_1 g(\cdot - c\tau) \|_{H^k}  \\
			&\qquad + \|\partial_1f(\cdot+\tau) g^2(\cdot - c\tau) \|_{H^k} + \| (f(\cdot + \tau) -f_+) \partial_1 g(\cdot - c\tau) \|_{H^k}\, \mathrm{d}\tau \\ 
			&\quad\leq C \int_0^T \sup_{x\in\R} \frac 1 {\langle x + \tau\rangle_+^2 \langle x - c\tau\rangle^2} \times \Bigg( \|f\|^2_{\mathcal X^{k+1}_{2}} \| g \|_{H^{k+1}_2} + \|f\|_{\mathcal X^{k+1}_{2}} \| g \|^2_{H^{k+1}_2} \Bigg) \, \mathrm d \tau, 
		\end{aligned}
	\end{equation}
	whence \cref{phi-bound} follows. The proof for \cref{psi-bound} is analogous.
\end{proof}

\section{Setup of Lattice Equations}\label{sec:lattice-eqns}

The scalar second-order differential equation \cref{fput-lattice-equations-strain-variables} with potential \(V\) given by \cref{potential} can be rewritten as the following first-order system:
\begin{equation}\label{first-order-lattice-eqns}
	\left\{\begin{aligned}\dot u _n &= q_{n+1} - q_n, \\
		\dot q_n &= u_n - u_{n-1} - \frac{1} 6 ( u_n^3 - u^3_{n-1}),\end{aligned} \right. \quad n \in \Z.
\end{equation}

Recall that \(u_{n} = x_{n+1} - x_n\), so we have that \(u_n\) physically represents the displacement between two neighbors on the lattice and \(q_n\) is equal to 
\begin{equation*}
	q_n(t) = \sum_{k=-\infty}^{n-1} \dot u_k(t) = \sum_{k=-\infty}^{n-1} [\dot x_{k+1}(t) - \dot x_k(t)] = \dot x_n(t)
\end{equation*}
and so represents the velocity at a lattice point (assuming that \(\dot x_k(t) \to 0\) as \(k\to-\infty\)). Note that we have the flexibility to add or subtract a constant from \(q\) without changing the dynamics on \(u\) (a fact that we use later to adjust the approximation and guarantee the error terms are in \(\ell^2(\Z)\)). 

\begin{rem}
	Writing the equations for the FPUT lattice in the form given by \cref{first-order-lattice-eqns} also puts the system into a Hamiltonian framework (when \(u, q\in\ell^2(\Z)\)). Here the equations are of the form
	\begin{equation*}
		\dot U = J \mcH'(U)
	\end{equation*}
	where \(U = (u,q)\), \(J\) is the skew-symmetric operator given by
	\begin{equation*}
		J = \begin{bmatrix}
			0 & e^\partial - 1 \\ 1 - e^{-\partial} & 0
		\end{bmatrix}
	\end{equation*}
	and \(\mcH(U) = \sum_{n\in\Z} \frac 1 2 q_n^2 + V(u_n)\). The operators \(e^\partial\) and \(e^{-\partial}\) are the forward and backward shift operators, respectively. So we have \((e^\partial u)_n = u_{n+1}\) and \((e^{-\partial}u)_n = u_{n-1}\).
\end{rem}

We will now introduce the small-amplitude, long-wavelength ansatz for the system in \cref{first-order-lattice-eqns}, but we first must assume certain regularity and decay of \(f\) and \(g\).
\begin{assum}\label{assumption-1}
	Let \(f\) and \(g\) be solutions of \cref{f-mKdV,g-gKdV}, respectively. Assume that \[f\in C([-\tau_0, \tau_0], \mcX_2^6(\R)) \quad \text{ and } \quad g\in C([-\tau_0,\tau_0],H_2^6(\R))\] for some \(\tau_0>0\) fixed. Furthermore, assume that \(f\) has fixed limits in its spatial variable at \(\pm \infty\) given by \(f_{\pm}\), respectively.
\end{assum}

\begin{rem}  While we are unaware of any general existence theorems for the equations \cref{f-mKdV,g-gKdV}
	in the weighted spaces $\mcX^k_2$ and $H^k_2$, we note that for \cref{f-mKdV} the kink solutions are examples of solutions that lie in $\mcX^k_2$ for all time and have non-zero limits at infinity.  Furthermore, from the stability results on kink solutions of mKdV, any solution with initial  conditions close to a kink will also remain in $\mcX^k_2$ for all time.  In addition, in \cite{Kappeler:2008}, while not working with exactly the same functions spaces we do,  the existence of  solutions of the mKdV equation with general non-zero limits at infinity is proven.  (In fact, the results of \cite{Kappeler:2008} even allow for solutions that grow at infinity.)  For equation \cref{g-gKdV} which governs the $g$-part of the approximation, we first note that this equation is a form of the Gardner equation which has soliton solutions, and hence again, there are families of solutions that remain in $H^k_2$ for all time.  Furthermore, since our results below require only existence of solutions for finite times, in these spaces, one can use the local well-posedness results for \cref{g-gKdV} in Sobolev spaces $H^s$ and $H^s_{2s}$, $s>3/2$,  given in \cite{Kato:1983}, to conclude that one has local well-posedness in $H^6_2$.
\end{rem}

The small-amplitude, long-wavelength ansatz for \(u_n\) and \(q_n\) is then given by
\begin{equation}\label{u-ansatz}
	u_n(t) = \epsilon f(\epsilon(n+t), \epsilon^3t) + \epsilon g(\epsilon(n-ct), \epsilon^3 t) + \epsilon^3 \phi(\epsilon n , \epsilon t) + \mcU_n(t)
\end{equation}
and 
\begin{equation}\label{q-ansatz}
	q_n(t) = \epsilon F(\epsilon(n+t), \epsilon^3 t) + \epsilon G(\epsilon(n-ct), \epsilon^3 t) + \epsilon^3 \Phi(\epsilon n, \epsilon t) - \epsilon F_- + \mcQ_n(t).
\end{equation}
The wave speed \(c\) is again given by \cref{ansatz-wave-speed}. 

The form that the ansatz takes for \(u_n(t)\) is clear. For \(q_n(t)\) we need to define \(F\), \(G\), and \(\Phi\) (where \(F_-\) is a constant to specified shortly thereafter). One would expect  
\begin{equation*}
	\begin{aligned}
		q_n(t) =& \sum_{k=-\infty}^{n-1} \dot u_n(t) \\
		\approx& \sum_{k=-\infty}^{n-1} [ \epsilon^2 \partial_1 f (\epsilon(k+t) ,\epsilon^3 t) + \epsilon^4 \partial_2 f(\epsilon(k+t), \epsilon^3 t) + \epsilon^2 c \partial_1 g(\epsilon (k-ct), \epsilon^3 t) +\epsilon^4 \partial_2 g(\epsilon(k-ct), \epsilon^3 t)  \\
		&\qquad + \epsilon^4 \partial_2 \phi(\epsilon k, \epsilon t) ].
	\end{aligned}
\end{equation*}
However, the final summation does not have a simple closed form, and so would be difficult to use. Instead, we use the fact that \(q_{n+1}(t) - q_n(t)= \dot u_n(t)\) and enforce equality with the ansatz for \(\dot u_n(t)\) up to sufficiently high orders of \(\epsilon\). Thus, we choose \(F\), \(G\), and \(\Phi\) so that 
\begin{equation*}
	\begin{aligned}
		\epsilon F(\epsilon (n+1 + t), \epsilon^3 t) - \epsilon F(\epsilon(n+t), \epsilon^3 t) &= \epsilon^2 \partial_1 f(\epsilon(n+t), \epsilon^3 t) + \epsilon^4 \partial_2 f(\epsilon (n+1), \epsilon^3 t) + \mathcal O(\epsilon^6) \\
		\epsilon G(\epsilon(n+1 -ct), \epsilon^3 t) - \epsilon G(\epsilon(n-ct), \epsilon^3 t) &=  \epsilon^2 c \partial_1 g(\epsilon (n-ct), \epsilon^3 t)  +\epsilon^4 \partial_2 g(\epsilon(n-ct), \epsilon^3 t) + \mathcal O(\epsilon^6) \\
		\epsilon^3 \Phi(\epsilon(n+1), \epsilon t) - \epsilon^3 \Phi(\epsilon n, \epsilon t) &= \epsilon^4 \partial_2 \phi(\epsilon n, \epsilon t) + \mathcal O(\epsilon^6) .
	\end{aligned}
\end{equation*}
Thus, after some calculations, we get the following:
\begin{align*}
	F &:= f - \frac{\epsilon} 2 \partial_1 f + \frac{\epsilon^2} 8 \partial_1^2 f - \frac{\epsilon^2}{12} f^3  - \frac{\epsilon^3}{48} \partial_1^3 f + \frac{\epsilon^3} 8 f^2 \partial_1 f \\
	G &:= - g + \frac{\epsilon}{2}\partial_1 g + \frac{\epsilon^2 f_+^2} 4  g + \frac{\epsilon^2}{12}(g^3 + 3f_+ g^2) -\frac{ \epsilon^2} 8 \partial_1^2 g  + \frac{\epsilon^3}{48} \partial_1^3 g  - \frac{\epsilon^3}{24} \partial_1(g^3 + 3f_+ g^2) - \frac{\epsilon^3 f_+^2} 8 \partial_1 g  \\
	\Phi &:=  \partial_1^{-1}\psi - \frac{\epsilon} 2 \psi.
\end{align*}
Here \(\psi = \partial_2 \phi\) and \(\partial_1^{-1}\) is defined as a Fourier multiplier. That \(\partial_1^{-1}\psi\) is well-defined and in \(H^5(\R)\) follows from \cref{psi-fourier-transform}. Namely, we have that 
\begin{equation*}
	\begin{aligned}
		\mcF[\partial_1^{-1} \psi(\cdot, T)](k) &= (ik)^{-1} \hat\psi(k,T) \\
		& = \frac{-i} 2 \int_0^T k \cos(k(T-\tau)) \times \\ 
		&\quad \mcF[(f^2(\cdot+\tau,\epsilon^2\tau)-f_+^2)g(\cdot-c\tau,\epsilon^2\tau) +(f(\cdot+\tau,\epsilon^2\tau)-f_+)g^2(\cdot-c\tau,\epsilon^2\tau)](k) \, d\tau
	\end{aligned}
\end{equation*}
and (following the same calculations in \cref{phi-sobolev-bound}) 
\begin{equation*}
	\| \partial_1^{-1}\psi(\cdot,T) \|_{H^5} \leq C \int_0^T \sup_{x\in\R} \frac 1 {\langle x + \tau\rangle_+^2 \langle x - c\tau\rangle^2} \times \Bigg( \|f\|^2_{\mathcal X^{6}_{2}} \| g \|_{H^{6}_2} + \|f\|_{\mathcal X^{6}_{2}} \| g \|^2_{H^{6}_2} \Bigg) \, \mathrm d \tau. 
\end{equation*}\Cref{assumption-1} implies that \(F\) has fixed limits in its spatial variable at \(\pm \infty\) given by \(F_\pm = f_\pm -\frac{\epsilon^2}{12} f^3_{\pm}\).

We want \(\mcU(t)\) and \(\mcQ(t)\) to be elements of \(\ell^2(\Z)\) (at least locally in time). However, to satisfy \(\mcQ(0)\in\ell^2(\Z)\) and \(\dot u_n(0) = q_{n+1}(0) - q_n(0)\), a compatibility condition must hold.
\begin{assum}\label{assumption-2}
	Assume that \[\sum_{n=-\infty}^\infty \dot u_n(0) = \epsilon F_+ - \epsilon F_-.\]
\end{assum}
Note that if this did not hold, then \(\mcQ_n(0)\not\to 0\) as \(n\to\infty\) and \(\mcQ(0)\notin \ell^2(\Z)\). That \(\mcQ_n(0) \to 0\) as \(n\to-\infty\) follows directly from the ansatz. The introduction of the constant \(\epsilon F_-\) in \cref{q-ansatz} does not affect the dynamics of \(q\) in \cref{first-order-lattice-eqns}

Substituting \cref{u-ansatz,q-ansatz} into \cref{first-order-lattice-eqns} and using the definitions of \(F\), \(G\), and \(\Phi\), we get the following equations for the \(\mcU_n\) and \(\mcQ_n\): 
\begin{equation}\label{error-lattice-eqns}
	\left\{\begin{aligned}
		\dot{\mathcal U}_n(t) =& \, \mathcal Q_{n+1}(t) - \mathcal Q_n(t) +\ResnI \\
		\dot{\mathcal Q}_n(t) =& \, \mathcal U_n(t) - \mathcal U_{n-1}(t)  \\
		&\quad - \frac 1 2 (\epsilon f(\epsilon(n+t)) + \epsilon g(\epsilon(n-ct)) + \epsilon^3\phi(\epsilon n))^2 \mathcal U_{n}(t) \\
		&\quad + \frac 12  (\epsilon f(\epsilon(n-1+t)) + \epsilon g(\epsilon(n-1-ct)) + \epsilon^3\phi(\epsilon (n-1)))^2 \mathcal U_{n-1}(t) \\
		&\quad +\ResnII + \mathcal B_n(\epsilon f + \epsilon g + \epsilon^3 \phi,  \mathcal U) 
	\end{aligned} \right. n\in\Z,
\end{equation}
where
\begin{equation*}
	\begin{aligned}
		\ResnI =& \epsilon F(\epsilon(n+1+t)) - \epsilon F(\epsilon(n+t)) \\
		&\quad + \epsilon G(\epsilon(n+1-c t) - \epsilon G(\epsilon(n-c t) + \epsilon^3 \Phi(\epsilon (n+1))   - \epsilon^3 \Phi(\epsilon n) \\
		&\quad - \epsilon^2 \partial_1 f(\epsilon(n+t)) - \epsilon^4 \partial_2
		f(\epsilon(n+t)) \\
		&\quad + \epsilon^2 c \partial_1 g(\epsilon(n-ct)) - \epsilon^4 \partial_2 g(\epsilon(n-ct)) - \epsilon^4 \partial_2 \phi(\epsilon n) ,
	\end{aligned}
\end{equation*}
\begin{equation*}
	\begin{aligned}
		\ResnII =& \epsilon f(\epsilon(n+t)) - \epsilon f(\epsilon(n-1+t)) \\
		&\quad+ \epsilon g(\epsilon(n-c t)) - \epsilon g(\epsilon(n-1-c t) + \epsilon^3 \phi(\epsilon n) - \epsilon^3 \phi(\epsilon (n-1)) \\
		&\quad - \epsilon^2 \partial_1 F(\epsilon(n+t)) - \epsilon^4 \partial_2
		F(\epsilon(n+t)) \\
		&\quad + \epsilon^2 c \partial_1 G(\epsilon(n-ct)) - \epsilon^4 \partial_2 G(\epsilon(n-ct)) - \epsilon^4 \partial_2 \Phi(\epsilon n) \\
		&\quad - \frac 1 6 \Big( (\epsilon f(\epsilon(n+t)) + \epsilon g(\epsilon(n-c t)) + \epsilon^3 \phi(\epsilon n))^3 \\
		&\hspace{6em} - (\epsilon f(\epsilon(n-1+t)) + \epsilon g(\epsilon(n-1-c t)) + \epsilon^3 \phi(\epsilon (n-1)))^3 \Big),
	\end{aligned}
\end{equation*}
and 
\begin{equation*}
	\begin{aligned}
		&\mathcal B_n(\epsilon f + \epsilon g + \epsilon^3 \phi, \mathcal U) \\
		&\quad= -\frac 1 6 \Big( 3(\epsilon f(\epsilon(n+t) + \epsilon g(\epsilon(n-ct)) + \epsilon^3 \phi(\epsilon n))  \mathcal U^2_n(t) \\
		&\quad \qquad - 3(\epsilon f(\epsilon(n-1+t) + \epsilon g(\epsilon(n-1-ct)) + \epsilon^3 \phi(\epsilon (n-1)))  \mathcal U^2_{n-1}(t) \\
		&\quad\qquad + \mathcal U_n^3(t)  - \mathcal U_{n-1}^3(t)\Big).
	\end{aligned}
\end{equation*}
The terms \(\mcU\) and \(\mcQ\) control the error associated with the ansatz in \cref{u-ansatz,q-ansatz}. Thus if these terms remain small in the \(\ell^2(\Z)\) norm, then the small-amplitude, long-wavelength ansatz will remain valid. In particular, if one has that \(\|\mcU\|_{\ell^2} \leq C\epsilon^{5/2}\), then the ansatz \(\epsilon f + \epsilon g\) is valid up to order \(\epsilon^{5/2}\) (since \(\phi\) is uniformly bounded in norm and is thus \(\mathcal O (1)\)). Similarly, if \(\mcQ\) is of order \(\epsilon^{5/2}\), then one can show that \(\dot u_n(t)\) is approximated by \(\epsilon^2 \partial_1 f + \epsilon^2 \partial_1 g\) up to order \(\epsilon^{5/2}\). Hence, controlling the norms of \(\mcU\) and \(\mcQ\) is sufficient in proving the approximation holds.

\section{Preparatory Estimates}\label{sec:estimates}

To control the dynamics of \(\mcU\) and \(\mcQ\), we need estimates of the residuals and the nonlinearity. We will frequently need to bound the \(\ell^2(\Z)\) norm of a term by the \(H^1(\R)\) norm of a function. To this end the following lemma proved in \cite{dumas2014justification} is useful.
\begin{lem}\label{h1-ell2-ineq}
	There exists \(C>0\) such that for all \(X \in H^1(\R)\) and \(\epsilon \in (0,1)\), \[\|x\|_{\ell^2} \leq C \epsilon^{-1/2} \|X\|_{H^1},\] where \(x_n := X(\epsilon n)\), \(n\in \mathbb Z\).
\end{lem}

We have the following estimates on the residual and nonlinear terms.
\begin{lem}\label{residual-nonlinearity-bounds}
	Let \(f\) and \(g\) be solutions of \cref{f-mKdV,g-gKdV}, respectively, such that \(f\in C([-\tau_0, \tau_0] , \mcX^6_2)\) and \(g\in C([-\tau_0,\tau_0], H^6_2)\). Let \(\tau_0 > 0\) be fixed and \(\delta>0\) be defined as \begin{equation}\label{delta-defn}
		\delta := \max \left\{\sup_{\tau\in[-\tau_0, \tau_0]}\|f(\cdot,\tau)\|_{\mcX^6_2},\ \sup_{\tau\in[-\tau_0, \tau_0]} \|g(\cdot, \tau)\|_{H^6_2} \right\}
	\end{equation}
	Then there exists a \(\delta\)-independent constant \(C>0\) such that the residual and nonlinear terms satisfy
	\begin{equation}\label{res-ineq}
		\| \ResI \|_{\ell^2} + \|\ResII \|_{\ell^2} \leq C \epsilon^{11/2} (\delta + \delta^5)
	\end{equation}
	and 
	\begin{equation}\label{nonlinear-ineq}
		\| \mathcal B_n(\epsilon f + \epsilon g + \epsilon^3 \phi, \mathcal U) \|_{\ell^2} \leq C\epsilon [ (\delta+\epsilon^2\delta^3) \|\mcU\|_{\ell^2} ^2 + \|\mcU\|_{\ell^2}^3]
	\end{equation}
	for every \(t\in[-\epsilon^{-3} \tau_0, \epsilon^{-3} \tau_0]\) and \(\epsilon \in (0,1).\)
\end{lem}

\begin{proof}
	We first focus on bounding \(\ResI\). Looking at the terms in \(\ResI\) involving \(f\) and \(F\) and using Taylor expansions and \cref{f-mKdV}, we get that
	\begin{equation}\label{F-res1}
			\epsilon F(\cdot + \epsilon) - \epsilon F - \epsilon^2 \partial_1 f - \epsilon^4 \partial_2 f = \epsilon^6 I_{f,1}(n,t)
	\end{equation}
	where terms of order \(\epsilon^5\) or lower exactly cancel. The term \(I_{f,1}\) contains integral remainder terms:
	\begin{equation}\label{If1}
		\begin{aligned}
			I_{f,1}(n,t) := \ &\frac{1} {24} \int_0^1 \partial_1^5 f(\epsilon(n+t+s), \epsilon^3 t)(1-s)^4\, ds - \frac{1} {12} \int_0^1 \partial_1^5 f(\epsilon(n+t+s), \epsilon^3 t)(1-s)^3\, ds \\
			+ & \frac{1} {16} \int_0^1 \partial_1^5 f(\epsilon(n+t+s), \epsilon^3 t)(1-s)^2\, ds -  \frac{1} {24} \int_0^1 \partial_1^3 (f^3)(\epsilon(n+t+s), \epsilon^3 t)(1-s)^2\, ds \\
			- & \frac{1}{48} \int_0^1 \partial_1^5f(\epsilon(n+t+s), \epsilon^3 t)(1-s)\, ds +  \frac{1}{24} \int_0^1 \partial_1^3 (f^3)(\epsilon(n+t+s), \epsilon^3 t) (1-s)\, ds.
		\end{aligned}
	\end{equation}
	Applying \cref{h1-ell2-ineq} (and \cref{prod-rule-1-lem,prod-rule-2-lem} when needed) to the terms in \cref{If1} gives that the \(\ell^2\) norm on the left-hand side of \cref{F-res1} can be bounded by \[C(\epsilon^{11/2}(\delta + \delta^3))\] for some choice of constant \(C>0\).
	
	Doing the same Taylor expansion for the \(g\) and \(G\) gives
	\begin{equation*}
			\epsilon G(\cdot + \epsilon) - \epsilon G + \epsilon^2c \partial_1 g - \epsilon^4 \partial_2 g  =  \epsilon^6 I_{g,1}(nt),
	\end{equation*}
	where the terms of order \(\epsilon^5\) and lower exactly cancel and \(I_{g,1}\) contains the integral remainder terms:
	\begin{equation}\label{Ig1}
		\begin{aligned}
			&I_{g,1}(n,t) := \\
			-&\frac{1} {24} \int_0^1 \partial_1^5 g(\epsilon(n - ct + s), \epsilon^3 t) (1-s)^4 \, ds +\frac{1} {12} \int_0^1 \partial_1^5 g(\epsilon(n - ct + s), \epsilon^3 t) (1-s)^3 \, ds \\
			+&\frac{ f_+^2} 8 \int_0^1 \partial_1^3 g(\epsilon(n - ct + s), \epsilon^3 t) (1-s)^2 \, ds +\frac{1} {24} \int_0^1 \partial_1^3(g^3)(\epsilon(n - ct + s), \epsilon^3 t) (1-s)^2 \, ds \\
			+&\frac{1} {24} \int_0^1 \partial_1^3(3f_+ g^2)(\epsilon(n - ct + s), \epsilon^3 t) (1-s)^2 \, ds - \frac{1} {16} \int_0^1 \partial_1^5g(\epsilon(n - ct + s), \epsilon^3 t) (1-s)^2 \, ds \\
			+ & \frac{1}{48}\int_0^1 \partial_1^5 g(\epsilon(n-ct+s), \epsilon^3 t)(1-s)\, ds -  \frac{1}{24} \int_0^1 \partial_1^3(g^3)(\epsilon(n-ct+s), \epsilon^3 t)(1-s)\, ds \\
			- & \frac{1}{24} \int_0^1 \partial_1^3(3f_+ g^2)(\epsilon(n-ct+s), \epsilon^3 t)(1-s)\, ds - \frac{f_+ ^2}{8} \int_0^1 \partial_1^3 g(\epsilon(n-ct+s), \epsilon^3 t) (1-s) \, ds
		\end{aligned}
	\end{equation}
	The terms in \cref{Ig1} can be controlled by \cref{h1-ell2-ineq}. Similarly we have
	\begin{equation*}
		\epsilon^3 \Phi(\epsilon(n+1), \epsilon t) - \epsilon^3\Phi(\epsilon n , \epsilon t) - \epsilon^4 \partial_2 \phi_2(\epsilon n, \epsilon t) =  \frac{\epsilon^6} 2 \int_0^1 \partial_1^2 \psi(\epsilon(n+s),\epsilon t)(1-s)^2\, ds,
	\end{equation*}
	so the \(\ell^2\) norm can also be controlled. Therefore we have 
	\begin{equation*}
		\| \ResI \|_{\ell^2} \leq C \epsilon^{11/2}(\delta + \delta^3)
	\end{equation*}
	
	The bound on \(\ResII\) can be approached similarly. Focusing on the terms with \(f\) and \(F\) in \(\ResII\), we have 
	\begin{equation}\label{F-res2}
			\epsilon f(\cdot) - \epsilon f(\cdot - \epsilon) - \epsilon^2 \partial_1 F -\epsilon^4\partial_2 F - \frac{\epsilon^3} 6 (f^3(\cdot) - f^3(\cdot - \epsilon)) = \epsilon^6 I_{f,2}(n,t).
	\end{equation}
	where the integral remainder terms are contained in \(I_{f,2}\):
	\begin{equation}\label{If2}
		\begin{aligned}
			I_{f,2}(n,t) := - & \frac{1} {24} \int_{0}^1 \partial_1^5 f (\epsilon(n+t+s), \epsilon^3 t)(s-1)^4\, ds \\
			+ & \frac{1} {12} \int_0^1 \partial_1^2(f^3)(\epsilon(n+t+s), \epsilon^3 t)(s-1)^2\, ds \\
			+ & \partial_2\left(\frac{1} 8 \partial_1^2 f - \frac{1}{12} f^3  - \frac{\epsilon}{48} \partial_1^3 f + \frac{\epsilon} 8 f^2 \partial_1 f\right)\left(\epsilon(n+t), \epsilon^3 t\right)
		\end{aligned}
	\end{equation}
	The integral terms in \cref{If2} can be controlled like before. The non-integral term can be controlled by first evaluating the derivative in time, \(\partial_2\), and replacing the terms \(\partial_2 f\) using \cref{f-mKdV}; then the terms can be controlled by \cref{h1-ell2-ineq}. Then the left-hand side of \cref{F-res2} can be bounded by a term of the form \[ C \epsilon^{11/2}(\delta + \delta^3).\]
	
	Taylor expanding the remaining terms in \(\ResII\) gives that they are equal to 
	\begin{equation*}
		\epsilon^4\left(-\partial_2 \partial_1^{-1} \psi+ \partial_1 \phi  - \frac 1 6 \partial_1(3(f^2 - f_+^2) g + 3(f-f_+) g^2)\right) + \epsilon^6 I_{g,2}(n,t)
	\end{equation*}
	where 
	\begin{equation}\label{Ig2}
		\begin{aligned}
			I_{g,2}(n,t) =& \\
			- & \frac{1} {24} \int_{0}^1 \partial_1^5 g(\epsilon(n-s-ct), \epsilon^3 t)  (s-1)^4 \, ds -  \frac{1}2 \int_0^1 \partial_1^3 \phi(\epsilon (n-s), \epsilon t) (s-1)^2\, ds \\
			- & \frac{ f_+^2}{4}\partial_1\left( \frac{f_+^2} 4  g + \frac{1}{12}(g^3 + 3f_+ g^2) -\frac{ 1} 8 \partial_1^2 g  + \frac{\epsilon}{48} \partial_1^3 g - \frac{\epsilon}{24} \partial_1(g^3 + 3f_+ g^2) - \frac{\epsilon f_+^2} 8 \partial_1 g\right) \\
			-&  \partial_2\left( \frac{f_+^2} 4  g + \frac{1}{12}(g^3 + 3f_+ g^2) -\frac{ 1} 8 \partial_1^2 g  + \frac{\epsilon}{48} \partial_1^3 g - \frac{\epsilon}{24} \partial_1(g^3 + 3f_+ g^2) - \frac{\epsilon f_+^2} 8 \partial_1 g\right) \\
			+ & \frac{1} {12} \int_0^1 \partial_1^3(g^3(\epsilon(n-s-ct), \epsilon^3 t)) (s-1)^2 ds  \\
			+ & \frac{1} {12} \int_0^1 \partial_1^3( 3g^2(\epsilon(n-s-ct), \epsilon^3 t)f(\epsilon(n-s+t), \epsilon^3 t)) (s-1)^2\, ds \\
			+ & \frac{1} {12} \int_0^1 \partial_1^3 (3g(\epsilon(n-s-ct), \epsilon^3 t)f^2(\epsilon(n-s+t), \epsilon^3 t)) (s-1)^2 \, ds
		\end{aligned}
	\end{equation}
	The other orders of \(\epsilon\) cancel exactly, as usual. The terms of order \(\epsilon^4\) are equal to 
	\begin{equation}\label{psi-phi-remainder}
		-\partial_2 \partial_1^{-1} \psi+ \partial_1 \phi  - \frac 1 6 \partial_1(3(f^2 - f_+^2) g + 3(f-f_+) g^2).
	\end{equation}
	Formally applying \(\partial_1\) implies that the above terms should be constant in space since \(\partial_2 \psi = \partial_2^2 \phi\) satisfies \cref{phi-pde}. However, one should be careful with this calculation due to the differences in scaling of the spatial variables: for example, \(\phi\) and \(\psi\)'s spatial variable is rescaled to \(\epsilon n\) while \(f\)'s is rescaled to \(\epsilon(n+t)\). Taking a derivative with respect to \(\xi = \epsilon x\) gives that \cref{psi-phi-remainder} must be constant. Since all the terms decay to zero at spatial infinity, \cref{psi-phi-remainder} is exactly zero. Thus \(\ResnII= \epsilon^6 I_{g,2}(n,t)\).
	
	The integral terms in \cref{Ig2} are bounded as before. The remaining terms in \cref{Ig2} can be bounded by evaluating \(\partial_2 g\) using \cref{g-gKdV} and then applying \cref{h1-ell2-ineq}. We can the get the following bound: \[\| \mathrm{Res}^{(2)}(t) \|_{\ell^2} \leq C \epsilon^{11/2} (\delta + \delta^3 + \delta^5).\] Interpolating between powers of \(\delta\) gives the desired inequality \cref{res-ineq}.
	
	The proof of \cref{nonlinear-ineq} follows immediately.
\end{proof}

To proceed, we construct an energy function for \cref{error-lattice-eqns} to control the \(\ell^2\) norms of \(\mcU\) and \(\mcQ\). \Cref{residual-nonlinearity-bounds} essentially states that \(\ResI\), \(\ResII\), and \(\mathcal B\) remain appropriately small. If one drops the residual and nonlinear terms from \cref{error-lattice-eqns}, then we are left with a linear (non-autonomous) Hamiltonian system. Hence, an appropriate choice of an energy function would simply be the Hamiltonian for this reduced system. Define 
\begin{equation}\label{energy-function}
	\mcE(t) = \frac 1 2 \sum_{n\in \mathbb Z} \mathcal Q_n^2(t) + \mathcal U_n^2(t) - \frac 1 2 \left(\epsilon f(\epsilon(n+t), \epsilon^3 t) + \epsilon g(\epsilon(n-ct, \epsilon^3t) + \epsilon^3 \phi(\epsilon n, \epsilon t)\right)^2 \mathcal U_n^2(t)
\end{equation}
The following lemma gives us that \(\mcE\) can be used to control \(\mcU\) and \(\mcQ\).
\begin{lem}\label{energy-coercive-bounds-lem}
	Fix \(\tau_0>0 \) and let \(\delta\) be given by \cref{delta-defn} . There exists \(\epsilon_0 = \epsilon_0(\delta) >0\) sufficiently small such that for every \(\epsilon \in (0,\epsilon_0)\) and for every local solution \((\mathcal U, \mathcal Q) \in C^1([-\tau_0\epsilon^{-3}, \tau_0\epsilon^{-3}], \ell^2(\mathbb Z))\) of \cref{error-lattice-eqns}, the energy-type quantity given in \cref{energy-function} is coercive with the bound
	\begin{equation}\label{coercive-bound}
		\|\mathcal Q(t) \|_{\ell^2}^2 + \| \mathcal U (t) \|_{\ell^2}^2 \leq  4 \mathcal E(t), \quad \text{for } t\in(-\tau_0\epsilon^{-3}, \tau_0\epsilon^{-3}).
	\end{equation}
	Moreover, there exists \(C> 0\) independent of \(\epsilon\) and \(\delta\) such that 
	\begin{equation*}
		\left|\frac{d\mathcal E}{dt} \right| \leq C \mathcal E^{1/2}\left[ \epsilon^{11/2} (\delta + \delta^5)  + \epsilon^3\delta^2\mathcal E^{1/2} + \epsilon(\delta + \mathcal{E}^{1/2})\mathcal E\right]
	\end{equation*}
	for every \(t\in [-\tau_0\epsilon^{-3}, \tau_0\epsilon^{-3}]\) and \(\epsilon \in (0,\epsilon_0)\). 	  
\end{lem}
\begin{proof}
	Note that \(\delta>0\) can be used to control the \(L^\infty(\R)\) norms of \(f\), \(g\), and \(\psi\). Thus we can choose \(\epsilon_0\) small enough so that for \(\epsilon \in (0,\epsilon_0)\) we have 
	\begin{equation*}
		1 - \frac 12 \left( \epsilon \| f\|_{L^\infty} + \epsilon \|g\|_{L^\infty} + \epsilon^3 \| \phi \|_{L^\infty} \right)^2 \geq \frac 12,
	\end{equation*}
	independent of the particular choices of \(f\) and \(g\). 	Hence
	\begin{equation*}
		\mathcal E(t) \geq \frac 1 2 \| \mathcal Q \|_{\ell^2}^2 + \frac 1 4 \| \mathcal U\|_{\ell^2}^2 \geq \frac 1 4\| \mathcal Q \|_{\ell^2}^2 + \frac 1 4 \| \mathcal U\|_{\ell^2}^2
	\end{equation*}
	and \cref{coercive-bound} follows.
	
	Now we take the time derivative of \(\mathcal E\) to get that 
	\begin{equation*}
		\begin{aligned}
			\frac{d\mcE}{dt} = \sum_{n \in \Z} &\mathcal Q_n(t) \ResnII + \mathcal Q_n(t) 	\mathcal B_n(\epsilon f + \epsilon g + \epsilon^3 \phi, \mathcal U(t)) \\
			&+ \mathcal U_n(t) \ResnI \left( 1 - \frac 12 (\epsilon f+ \epsilon g + \epsilon^3 \phi)^2  \right) \\
			&+ \mathcal U_n^2(t)(\epsilon f + \epsilon g + \epsilon^3 \phi) \times (\epsilon^2\partial_1 f + \epsilon^4\partial_2 f -\epsilon^2 c \partial_1 g + \epsilon^4 \partial_2 g + \epsilon^4 \partial_2 \phi).
		\end{aligned}
	\end{equation*}
	Then using the Cauchy inequality and the H\"older inequality for \(p=1\) and \(q=\infty\) we get that
	\begin{equation*}
		\begin{aligned}
			\left| \frac{d\mathcal E}{dt} \right| \leq& \| \mathcal Q \|_{\ell^2 }\times \|\ResII \|_{\ell^2} + \| \mathcal Q \|_{\ell^2} \times \| \mathcal B \|_{\ell^2}  + \|\mathcal U \|_{\ell^2} \times \| \ResnI \|_{\ell^2}\\
			& + \|\mathcal U^2 \|_{\ell^1} \times  \| (\epsilon f+ \epsilon g + \epsilon^3 \phi)  \times  (\epsilon^2\partial_1 f + \epsilon^4\partial_2 f -\epsilon^2 c \partial_1 g + \epsilon^4 \partial_2 g + \epsilon^4 \partial_2 \phi )\|_{\ell^\infty} .
		\end{aligned}
	\end{equation*}
	Note that if \(a \in \ell^2\), then \(a\in \ell^\infty\) and \(\|a\|_{\ell^\infty} \leq \|a \|_{\ell^2}\). Thus we can replace the \(\ell^\infty\) norms above with \(\ell^2\) norms. Using the results in \cref{residual-nonlinearity-bounds}, we thus have 
	\begin{equation*}
		\begin{aligned}
			\left| \frac{d\mathcal E}{dt} \right| \leq& C\Big[\mathcal E^{1/2} \epsilon^{11/2}(\delta + \delta^5) + \mathcal E^{1/2}\epsilon [(\delta + \epsilon^2 \delta^3)\mathcal E + \mathcal E^{3/2}]  \\
			&\quad + \mathcal E(\epsilon^3 \delta^2 + \epsilon^5\delta^2 + \epsilon^5 \delta^4 + \epsilon^7\delta^4 + \epsilon^7 \delta^6) \Big],
		\end{aligned}
	\end{equation*}
	where the \(C>0\) is independent of \(\epsilon\) and \(\delta\). The right-hand side of the above inequality can be simplified by taking \(\epsilon_0\) smaller. That is, taking \(\epsilon_0\) sufficiently small (dependent on \(\delta\)), we can absorb higher orders of \(\epsilon\) into lower orders. For example, \(\epsilon^3 \delta^2 + \epsilon^5\delta^2 \leq 2 \epsilon^3 \delta^2\) for \(\epsilon\) small enough. Thus we arrive at  		
	\begin{equation*}
		\left|\frac{d\mathcal E}{dt} \right| \leq C \mathcal E^{1/2}\left[ \epsilon^{11/2} (\delta + \delta^5)  + \epsilon^3\delta^2\mathcal E^{1/2} + \epsilon(\delta + \mathcal{E}^{1/2})\mathcal E\right]
	\end{equation*}
	as desired.
\end{proof}

Lastly, before we can prove our main result, we must show that for appropriate initial conditions that \(\mcU(0)\) and \(\mcQ(0)\) are suitably small. In particular, we want our initial conditions to be ``close to" the small-amplitude, long-wavelength ansatz in the sense that 
\begin{equation*}
	u_n(0) \approx \epsilon f(\epsilon n , 0) + \epsilon g(\epsilon n , 0)
\end{equation*}
and 
\begin{equation*}
	\dot u_n(0) \approx \epsilon^2 \partial_1 f(\epsilon n , 0) -\epsilon^2 \partial_ 1 g(\epsilon n,0)
\end{equation*}
where the higher-order \(\epsilon\) terms are neglected. Recall that we assume \(\phi\) and \(\partial_1\phi\) to have initial conditions exactly equal to zero, so those terms drop. A seemingly appropriate notion of ``closeness" would be in the \(\ell^2\) norm, as used in \cites{khan2017long,schneider2000counter}. However, since \(q_n(0) = \sum_{k=-\infty}^{n-1} \dot u_{k}(0)\), we may lose some decay due to the summation and \(\mcQ(0)\) will not be in \(\ell^2\). To counter this, we need some extra localization assumptions on \(\dot u_n(0)\).

\begin{assum}\label{assumption-3}
	Suppose that the initial conditions for \(u\) satisfy
	\begin{equation*}
		\| u(0) - \epsilon f(\epsilon \cdot, 0) - \epsilon g(\epsilon \cdot, 0) \|_{\ell^2} + \| \dot u(0) - \epsilon^2 \partial_1 f(\epsilon \cdot, 0) + \epsilon^2 \partial_1 g (\epsilon \cdot , 0) \|_{\ell^2_2} \leq \epsilon^{5/2}
	\end{equation*}
	and that \(f(\cdot, 0) \in \mcX^6_2\) and \(g(\cdot, 0) \in H^6_2\)
\end{assum}

The \(\ell^2_2\) norm (defined by \(\|a\|_{\ell^2_2} = \| \langle n \rangle^2 a_n\|_{\ell^2}\)) will be sufficient to get that the summation is in \(\ell^2\) based on the following lemma.
\begin{restatable}{lem}{elltwo}
	\label{ell22-lemma}
	If \(a\in \ell^2_2(\Z)\) and 
	\begin{equation*}
		\sum_{k=-\infty}^\infty a_k = 0,
	\end{equation*}
	then \(b_n = \sum_{k=-\infty}^n a_k\) is in \(\ell^2(\Z)\) and 
	\begin{equation*}
		\|b\|_{\ell^2} \leq C \|a\|_{\ell^2_2}
	\end{equation*}
	for some \(C> 0\) independent of \(a\).
\end{restatable} 
See \cref{lemma-appendix} for proof.

We can now show the following.
\begin{lem}\label{initial-conditions-lem}
	Let \cref{assumption-2,assumption-3} hold. Then \(\mcU(0),\mcQ(0) \in \ell^2(\Z)\) satisfy
	\begin{equation}\label{initial-condition-ineq}
		\| \mcU(0) \|_{\ell^2} + \|\mcQ(0)\|_{\ell^2} \leq C \epsilon^{5/2}
	\end{equation}
	with \(C>0\) independent of \(\epsilon\).
\end{lem}

\begin{proof}
	That \(\|\mcU(0)\|_{\ell^2}\leq C \epsilon^{5/2}\) follows immediately from applying \cref{assumption-3} to \cref{u-ansatz}.
	
	We have from the definition of \(u_n\) and \(q_n\) that
	\begin{equation}\label{diff-eqn}
		\dot u_n(0) = q_{n+1}(0) - q_n(0)
	\end{equation}
	For \(q_n(0)\) to satisfy \cref{diff-eqn}, it must equal \(\sum_{k=-\infty}^{n-1} \dot u_k(0)\) (modulo a constant which we assume without loss of generality to be zero). Thus we have
	\begin{equation}\label{q0-summation}
		\begin{aligned}	
			q_n(0) =& \sum_{k=-\infty}^{n-1} \dot u_k(0)\\
			=& \sum_{k=-\infty}^{n-1}\left[ \dot u_k(0) - \epsilon^2 \partial_1 f(\epsilon k,0) - \epsilon^4 \partial_1 f(\epsilon k,0) + \epsilon^2 c \partial_1 g(\epsilon k,0) - \epsilon^4 \partial_2g(\epsilon k,0)\right] \\
			&+\sum_{k=-\infty}^{n-1}\left[  \epsilon^2 \partial_1 f(\epsilon k,0) +\epsilon^4 \partial_1 f(\epsilon k,0) - \epsilon F(\epsilon(k+1),0) +\epsilon F(\epsilon k ,0)  \right] \\
			&+ \sum_{k=-\infty}^{n-1}\left[ - \epsilon^2 c\partial_1 g(\epsilon k,0) +\epsilon^4 \partial_1 g(\epsilon k,0) - \epsilon G(\epsilon(k+1),0) +\epsilon G(\epsilon k ,0)  \right] \\
			&+ \epsilon F(\epsilon n, 0) - \epsilon F_- + \epsilon G(\epsilon n, 0).
		\end{aligned}
	\end{equation}
	Comparing \cref{q0-summation} to \cref{q-ansatz}, we have that 
	\begin{equation*}\label{mcq-zero}
		\begin{aligned}
			\mcQ_n(0) =& \sum_{k=-\infty}^{n-1}\left[ \dot u_k(0) - \epsilon^2 \partial_1 f(\epsilon k,0) - \epsilon^4 \partial_1 f(\epsilon k,0) + \epsilon^2 c \partial_1 g(\epsilon k,0) - \epsilon^4 \partial_2g(\epsilon k,0)\right] \\
			&+\sum_{k=-\infty}^{n-1}\left[  \epsilon^2 \partial_1 f(\epsilon k,0) +\epsilon^4 \partial_1 f(\epsilon k,0) - \epsilon F(\epsilon(k+1),0) +\epsilon F(\epsilon k ,0)  \right] \\
			&+ \sum_{k=-\infty}^{n-1}\left[ - \epsilon^2 c\partial_1 g(\epsilon k,0) +\epsilon^4 \partial_1 g(\epsilon k,0) - \epsilon G(\epsilon(k+1),0) +\epsilon G(\epsilon k ,0)  \right].
		\end{aligned}
	\end{equation*}
	That \(\mcQ_n(0)\to 0\) as \(n\to\infty\) is guaranteed by \cref{assumption-2}. Now \cref{ell22-lemma} can be applied to get the result if the summands are in \(\ell^2_2\) and of order \(\epsilon^{5/2}\). The first summand satisfies this condition because of \cref{assumption-3}.  Note that the latter summands are equal to \(-\epsilon^6 I_{f,1}(k,0)\) and \(-\epsilon^6 I_{g,1}(k,0)\), as defined in \cref{If1,Ig1}. This follows from the earlier calculations in \cref{residual-nonlinearity-bounds}. That \(\epsilon^6 I_{f,1}(k,0)\) and \(\epsilon^6 I_{g,2}(k,0)\) are elements of \(\ell^2_2\) follows from \(f(\cdot, 0) \in \mcX^6_2\) and \(g(\cdot, 0) \in H^6_2\) and an application of \cref{h1-ell2-ineq}.
	
	Thus we have \cref{initial-condition-ineq} where the \(C>0\) can be chosen based on the norms of \(f\) and \(g\).
\end{proof}

\section{Long-time approximation of FPUT}\label{sec:long-time-approx}

In this section, we show that any phenomenon observed in the pair of counter propagating KdV equations, given by \cref{f-mKdV,g-gKdV}, can also be observed in the FPUT lattice \cref{fput-lattice-equations-strain-variables}. In particular, our result is as follows.

\begin{theorem}\label{thm:long-time-stability}
	Let \cref{assumption-1} hold and set
	\begin{equation*}
		\delta = \max \left\{\sup_{\tau\in[-\tau_0, \tau_0]}\|f(\cdot,\tau)\|_{\mcX^6_2},\ \sup_{\tau\in[-\tau_0, \tau_0]} \|g(\cdot, \tau)\|_{H^6_2} \right\}
	\end{equation*}
	There exists positive constants \(\epsilon_0\) and \(C\) such that for all \(\epsilon \in(0,\epsilon_0)\), when initial data \((u(0), \dot u(0))\) satisfy \cref{assumption-2,assumption-3}, the unique solution \((u,q)\) to the FPU equation \cref{first-order-lattice-eqns} belongs to 
	\begin{equation*}
		C^1([-t_0(\epsilon), t_0(\epsilon)], \ell^\infty(\mathbb Z))
	\end{equation*}
	with \(t_0(\epsilon):= \epsilon^{-3}\tau_0 \) and satisfies
	\begin{equation*}
		\begin{aligned}
			&\| u(t) - \epsilon f(\epsilon(\cdot+t), \epsilon^3 t) -\epsilon g(\epsilon(\cdot -ct) ,\epsilon^3 t) \|_{\ell^2} \\
			&\quad + \| \dot u(t) - \epsilon \partial_1 f(\epsilon (\cdot +t),\epsilon^3t)  +\epsilon^2 \partial_1 g(\epsilon(\cdot - ct), \epsilon^3t)\|_{\ell^2} \leq C \epsilon^{5/2 }, \quad t\in[-t_0(\epsilon), t_0(\epsilon)].
		\end{aligned}
	\end{equation*}
\end{theorem}

\begin{proof}
	Set \(\mathcal S := \mathcal E ^{1/2}\) where \(\mathcal E\) is defined in \cref{energy-function}. From the results in \cref{initial-conditions-lem}, we get that \(\mathcal S(0) \leq C_0 \epsilon^{5/2}\) for some constant \(C_0 > 0\) and \(\epsilon_0\) as chosen in \cref{energy-coercive-bounds-lem}. For fixed constant \(C> 0\) define
	\begin{equation*}
		T_{C} := \sup \left\{T_0 \in (0,   \epsilon^{-3} \tau_0]: \mathcal S(t) \leq C \epsilon^{5/2},\, t\in [-T_0, T_0]\right\}.
	\end{equation*} 
	The goal is then to pick \(C\) so that \(T_{C} = \epsilon^{-3} \tau_0\).
	
	We have that
	\begin{equation*}
		\begin{aligned}
			\left | \frac d {dt} \mathcal S(t) \right | &= \frac 1 {2 \mathcal E ^{1/2}} \left | \frac d {dt} \mathcal E(t) \right| \\
			&\leq C_1(\delta + \delta^5) \epsilon^{11/2} + C_2 \epsilon^3\left[ \delta^2 + \epsilon^{-2}(\delta + \mathcal S) \mathcal S \right]\mathcal S
		\end{aligned}
	\end{equation*}
	where \(C_1, C_2 > 0\) are independent of \(\delta\) and \(\epsilon\). While \(|t| \leq T_{C}\),
	\begin{equation*}
		C_2 \left[ \delta^2 + \epsilon^{-2}(\delta + \mathcal S) \mathcal S \right] \leq C_2 \left[ \delta^2  + (\delta +  C\epsilon^{5/2}) C \epsilon^{1/2} \right],
	\end{equation*}
	where the right-hand side is continuous in \(\epsilon \) for \(\epsilon \in [0,\epsilon_0]\) and \(C>0\). Furthermore, the right-hand side of the inequality above is increasing in both \(\epsilon\) and \(C\), and so we can uniformly bound the term by some fixed number. Set \(K(C,\epsilon_0)= K>0\) to be
	\begin{equation*}\label{K-def-2}
		K :=  \left[ \delta^2  + (\delta +  C\epsilon_0^{5/2}) C \epsilon_0^{1/2} \right].
	\end{equation*}
	
	Hence, we can get that for \(t \in [-T_{C}, T_{C}]\) 
	\begin{equation*}
		\begin{aligned}
			\frac d {dt} e^{-\epsilon^3 K t} \mathcal S(t) &= - \epsilon^3 K e^{-\epsilon^3 K t} \mathcal S  + e^{-\epsilon^3 K t} \frac d {dt} \mathcal S \\
			&\leq - \epsilon^3 K e^{-\epsilon^3 K t} \mathcal S  + e^{-\epsilon^3 K t}C_1(\delta + \delta^5) \epsilon^{11/2} \\
			&\qquad+ e^{-\epsilon^3 K t}C_2 \epsilon^3\left[ \delta^2 + \epsilon^{-2}(\delta + \mathcal S) \mathcal S \right]\mathcal S \\
			&\leq - \epsilon^3 K e^{-\epsilon^3 K t} \mathcal S  +  e^{-\epsilon^3 K t}C_1(\delta + \delta^5) \epsilon^{11/2} + \epsilon^3 K e^{-\epsilon^3 K t}\mathcal S \\
			&= e^{-\epsilon^3 K t}C_1(\delta + \delta^5) \epsilon^{11/2}.
		\end{aligned}
	\end{equation*}
	Integrating gives
	\begin{equation*} 
		\begin{aligned}
			\mathcal S(t) &\leq \left( \mathcal S(0) + K^{-1} C_1 (\delta+\delta^5) \epsilon^{5/2} \right) e^{\epsilon^3 K t} - \epsilon^{5/2} K^{-1} C_1 (\delta + \delta^5) \\
			&\leq \left(C_0+ K^{-1} C_1 (\delta+\delta^5)  \right) \epsilon^{5/2}e^{\epsilon^3 K t} \\
			&\leq (C_0 + K^{-1} C_1 (\delta+\delta^5) ) e^{K\tau_0} \epsilon^{5/2}
		\end{aligned}
	\end{equation*}
	for \(t \in [-T_{C}, T_{C}]\). If we have 
	\begin{equation}\label{C-ineq-bound}
		(C_0 + K^{-1} C_1 (\delta+\delta^5) ) e^{K\tau_0}  \leq C
	\end{equation} 
	then we can conclude that \(T_C = \epsilon^{-3}\tau_0.\) Note that the left-hand side of the inequality goes to 
	\begin{equation*}
		(C_0 + \delta^{-2} C_1 (\delta+\delta^5) ) e^{\delta^2\tau_0}
	\end{equation*}
	as \(\epsilon \to 0\) for fixed values of \(C\). Thus choose \(C>0\) large enough so that
	\begin{equation*}
		(C_0 + \delta^{-2} C_1 (\delta+\delta^5) ) e^{\delta^2\tau_0} < C
	\end{equation*}
	and then we can make \(\epsilon_0\) sufficiently small so that \cref{C-ineq-bound} holds for all \(\epsilon \in (0,\epsilon_0]\).
\end{proof}

\section{Meta-stability of kink-like solutions}\label{sec:meta}

We would now like to apply a similar method as seen in \cite{khan2017long} to show that the approximations hold for time scales of order \(\mcO(\epsilon^{-3}|\log(\epsilon)|)\). This is a useful result because one can then make conclusions about the meta-stability of the kink-like solution on the FPUT from the stability of the kink solution for the mKdV. 

However, we cannot use the full approximation with the counter-propagating solutions. The problem comes from trying to extend \cref{assumption-1}. To make sure \(\phi\) remains bounded for longer period of times, we need to assume that \(f\) and \(g\) remain localized for longer and longer times. However, the PDEs \cref{f-mKdV} and \cref{g-gKdV} are dispersive, and so generic solutions will become less localized over time resulting in larger norms in \(\mcX^6_2\) and \(H^6_2\).

The localization assumption is only necessary to keep \(\phi\), the term coming from the coupling of \(f\) and \(g\), bounded. We can drop this assumption if we set \(g\) identically equal to zero. It is easy to see that if \(g=0\) then \(\phi = 0\). Also, one can check that the estimates of the residuals and nonlinear terms rely only on \(f \in \mcX^6\) if \(\phi = 0\), and so our estimates from before still hold in this case.

\begin{assum}\label{assumption-4}
	Let \(f\) be a solution to \cref{f-mKdV} and set \(g = 0\). Assume that 
	\begin{equation*}
		f \in C_b(\R, \mcX^6(\R)).
	\end{equation*}
	Furthermore, assume that \(f\) has fixed limits in its spatial variables at \(\pm \infty\) given by \(f_{\pm}\).
\end{assum}

We will still need to assume that the initial condition of \(f\) is localized as in \cref{assumption-3}, but this assumption holds for many solutions including the kink solutions of \cref{f-mKdV}.

The following result and proof are analogous to those of \cite{khan2017long}*{Thm.~1}. The idea behind the proof is to sacrifice some accuracy in the approximation (so that the error is \(\mcO(\epsilon^{5/2 - r})\)) in order to extend the time that the approximation holds (which will now be \(\mcO(\epsilon^{-3} |\log(\epsilon)|)\)).

\begin{theorem}\label{thm:meta-stable}
	Let \cref{assumption-4} hold and set 
	\begin{equation*}
		\delta =\sup_{\tau\in\R}\|f(\cdot,\tau)\|_{\mcX^6}
	\end{equation*}
	For fixed \(r\in(0,1/2)\), there exists positive constants \(\epsilon_0\), \(C\), and \(K\) such that for all \(\epsilon \in(0,\epsilon_0)\), when initial data \((u(0), \dot u(0))\) satisfy \cref{assumption-2,assumption-3}, the unique solution \((u,q)\) to the FPU equation \cref{first-order-lattice-eqns} belongs to 
	\begin{equation*}
		C^1([-t_0(\epsilon), t_0(\epsilon)], \ell^\infty(\mathbb Z))
	\end{equation*}
	with \(t_0(\epsilon):= r K^{-1} \epsilon^{-3} | \log (\epsilon) | \) and satisfies
	\begin{equation*}
		 \| u(t) - \epsilon f(\epsilon(\cdot+t), \epsilon^3 t) \|_{\ell^2} + \| \dot u(t) - \epsilon \partial_1 f(\epsilon (\cdot +t),\epsilon^3t)  \|_{\ell^2} \leq C \epsilon^{5/2 - r}, \quad t\in[-t_0(\epsilon), t_0(\epsilon)].
	\end{equation*}
\end{theorem}

\begin{proof}
	
	Set \(\mathcal S := \mathcal E ^{1/2}\) where \(\mathcal E\) is defined in \cref{energy-function}. From the results in \cref{initial-conditions-lem}, we get that \(\mathcal S(0) \leq C_0 \epsilon^{5/2}\) for some constant \(C_0 > 0\) and \(\epsilon_0\) as chosen in \cref{energy-coercive-bounds-lem}. For fixed constants \(r\in(0,1/2)\), \(C> C_0\), and \(K > 0\), define the maximal continuation time by 
	\begin{equation*}
		T_{C,K,r} := \sup \left\{T_0 \in (0, r K^{-1} \epsilon^{-3} |\log(\epsilon)|]: \mathcal S(t) \leq C \epsilon^{5/2 -r}, t\in [-T_0, T_0]\right\}.
	\end{equation*} 
	We also define the maximal evolution time of the mKdV equation as \(\tau_0(\epsilon) = rK^{-1}|\log(\epsilon)|\). The goal is then to pick \(C\) and \(K\) so that \(T_{C,K,r} = \epsilon^{-3} \tau_0(\epsilon)\).
	
	We have that
	\begin{equation*}
		\begin{aligned}
			\left | \frac d {dt} \mathcal S(t) \right | &= \frac 1 {2 \mathcal E ^{1/2}} \left | \frac d {dt} \mathcal E(t) \right| \\
			&\leq C_1(\delta + \delta^5) \epsilon^{11/2} + C_2 \epsilon^3\left[ \delta^2 + \epsilon^{-2}(\delta + \mathcal S) \mathcal S \right]\mathcal S
		\end{aligned}
	\end{equation*}
	where \(C_1, C_2 > 0\) are independent of \(\delta\) and \(\epsilon\). While \(|t| \leq T_{C,K,r}\),
	\begin{equation*}
		C_2 \left[ \delta^2 + \epsilon^{-2}(\delta + \mathcal S) \mathcal S \right] \leq C_2 \left[ \delta^2  + \epsilon^{-2}(\delta +  C\epsilon^{5/2-r}) C \epsilon^{5/2-r} \right],
	\end{equation*}
	where the right-hand side is continuous in \(\epsilon \) for \(\epsilon \in [0,\epsilon_0]\). Thus the right-hand side can be uniformly bounded by a constant independent of \(\epsilon\). Choose \(K>0\) (dependent on \(C\)) sufficiently large so that 
	\begin{equation}\label{K-def}
		C_2 \left[ \delta^2  + \epsilon^{-2}(\delta +  C\epsilon^{5/2-r}) C \epsilon^{5/2-r} \right] \leq K.
	\end{equation}
	
	Hence, we can get that for \(t \in [-T_{C,K,r}, T_{C,K,r}]\) 
	\begin{equation*}
		\begin{aligned}
			\frac d {dt} e^{-\epsilon^3 K t} \mathcal S(t) &= - \epsilon^3 K e^{-\epsilon^3 K t} \mathcal S  + e^{-\epsilon^3 K t} \frac d {dt} \mathcal S \\
			&\leq - \epsilon^3 K e^{-\epsilon^3 K t} \mathcal S  + e^{-\epsilon^3 K t}C_1(\delta + \delta^5) \epsilon^{11/2} \\
			&\qquad+ e^{-\epsilon^3 K t}C_2 \epsilon^3\left[ \delta^2 + \epsilon^{-2}(\delta + \mathcal S) \mathcal S \right]\mathcal S \\
			&\leq - \epsilon^3 K e^{-\epsilon^3 K t} \mathcal S  +  e^{-\epsilon^3 K t}C_1(\delta + \delta^5) \epsilon^{11/2} + \epsilon^3 K e^{-\epsilon^3 K t}\mathcal S \\
			&= e^{-\epsilon^3 K t}C_1(\delta + \delta^5) \epsilon^{11/2}.
		\end{aligned}
	\end{equation*}
	Integrating gives
	\begin{equation*} 
		\begin{aligned}
			\mathcal S(t) &\leq \left( \mathcal S(0) + K^{-1} C_1 (\delta+\delta^5) \epsilon^{5/2} \right) e^{\epsilon^3 K t} - \epsilon^{5/2} K^{-1} C_1 (\delta + \delta^5) \\
			&\leq \left( \mathcal S(0) + K^{-1} C_1 (\delta+\delta^5) \epsilon^{5/2} \right) e^{\epsilon^3 K t} \\
			&\leq \left( \mathcal S(0) + K^{-1} C_1 (\delta+\delta^5) \epsilon^{5/2} \right) e^{ K \tau_0(\epsilon)} \\
			&\leq \left( C_0 + K^{-1} C_1 (\delta+\delta^5)  \right) \epsilon^{5/2 -r}
		\end{aligned}
	\end{equation*}
	for \(t \in [-T_{C,K,r}, T_{C,K,r}]\), where the last line follows in part from the definition of \(\tau_0(\epsilon)\). Now choose \(C> C_0\) sufficiently large so that 
	\begin{equation*}
		C_0 + K^{-1} C_1(\delta + \delta^5) \leq C.
	\end{equation*}
	Note that our earlier choice of \(K\) can be enlarged so that \cref{K-def} still holds as well as the above inequality. Therefore, with these choices of \(C\) and \(K\), the maximal interval can be extended to \(T_{C,K,r} = \epsilon^{-3} \tau_0(\epsilon) \). 
\end{proof}

One might assume that the earlier result in \cref{thm:long-time-stability} can be used to determine stability results of the kink-like solution of the FPUT by using stability results of the kink solution of the defocusing mKdV. It has been previously shown that the kink solution of the defocusing mKdV is \(H^1\) orbitally stable \cite{zhidkov2001korteweg}*{Thm.~III.2.4}, as well as \(H^1_{\mathrm{loc}}\) asymptotically stable \cite{merle20032}*{Cor. ~4.4}. However applying these results directly with this theorem leads to an unsatisfying answer. Since the approximation holds only for times \([-\epsilon^{-3} \tau_0, \epsilon^{-3} \tau_0]\) for the FPUT, we are only looking at a fixed slice of time \([-\tau_0, \tau_0 ]\) for the solution \(f\) of the mKdV. Thus regardless of the stability results, the solution \(f\) will remain a fixed distance from the kink solution for any choice of \(\epsilon > 0\). In particular, \(f\) would not be approaching the kink solution, and so we would not have any corresponding asymptotic stability result for the kink-like solution.

\Cref{thm:meta-stable} can be applied to discuss stability results of the kink-like solutions of the FPUT. This result extends the approximation beyond the natural time scale \(\mcO(\epsilon^{-3})\) that one would naively expect to the larger time scale \(\mcO(\epsilon^{-3} |\log(\epsilon)|)\), so that we may observe \(f\) on a slice of time of order \(|\log(\epsilon)|\). So if \(f\) approaches the kink solution in \(H^1_{\mathrm{loc}}\), then the ansatz should approach the solution given by the kink on the FPUT, given \(\epsilon>0\) is chosen small enough. Furthermore, since the ansatz from the kink solution remains close to the kink-like solution found in \cite{norton2023kinklike}, we have the desired meta-stability result for the kink-like solutions.

\appendix

\section{Proofs of lemmas}\label{lemma-appendix}

\begin{proof}[Proof of \cref{prod-rule-1-lem}]
	The result follows from induction on \(k\).
	
	For \(k = 0\), we have
	\begin{equation}
		\| f g \|_{H^0} \leq \| f \|_{L^\infty} \| g\|_{H^0}.
	\end{equation}
	
	Assuming \cref{prod_rule} holds for \(k\geq 0\), we have that 
	\begin{align*}
		\| f g \|_{H^{k+1}} \leq C \left( \| f g \|_{H^k} + \| \partial^{k+1}(fg) \|_{L^2}\right) \\
		\leq C \left( \| f\|_{\mathcal X^k} \| g \|_{H^k} + \| \partial^{k+1}(fg) \|_{L^2} \right),
	\end{align*}
	where the second term can be bounded by 
	\begin{align*}
		\| \partial^{k+1}(fg) \|_{L^2} &\leq \| \partial^k(f' g ) \|_{L^2} + \| \partial^k(f g') \|_{L^2} \\
		&\leq \| f' g \|_{H^k} + \| f g' \|_{H^k} \\
		&\leq \| f'\|_{H^k} \|g\|_{H^k} + \|f\|_{\mathcal X^k} \|g'\|_{H^k} \\
		&\leq \| f \|_{\mathcal X^{k+1}} \| g\|_{H^{k+1}} + \|f \|_{\mathcal X^{k+1}} \|g\|_{H^{k+1}} \\
		&= 2  \| f \|_{\mathcal X^{k+1}} \| g\|_{H^{k+1}}.
	\end{align*}
	This completes the induction.
\end{proof}

\begin{proof}[Proof of \cref{prod-rule-2-lem}]
	Using the result from \cref{prod-rule-1-lem}, we have 
	\begin{equation}
		\begin{aligned}
			\| f g \|_{\mcX^k} &\leq \| f g\|_{L^\infty} + \| (fg)' \|_{H^{k-1}} \\
			&\leq \|f \|_{L^\infty} \| g \|_{L^\infty} + \|f'g \|_{H^{k-1}} + \| fg' \|_{H^{k-1}} \\
			&\leq \|f \|_{L^\infty} \| g \|_{L^\infty} + C\|f'\|_{H^{k-1}} \|g \|_{\mcX^{k-1}} + \| f \|_{\mcX^{k-1}} \|g'\|_H^{k-1} \\
			&\leq C \|f \|_{\mcX^k} \|g\|_{\mcX^k} .
		\end{aligned}
	\end{equation}
\end{proof}

\begin{proof}[Proof of \cref{Ck-bound}]
	The main argument of the proof is given by showing the following claim holds:
	
	\emph{Claim}: For each integer \(k\geq 0\), \[\frac{\partial^k}{\partial x^k} \left[\frac 1 {\langle x+\tau\rangle_+^2 \langle x -c\tau \rangle^2}\right]\] is a sum of terms of the form 
	\begin{equation}\label{term}
		\frac C {\langle x +\tau\rangle_+^{2+m} \langle x -c\tau \rangle^{2+m}}\langle x + \tau \rangle_+^{m_1}\langle x-c\tau\rangle^{m_2} F(x,\tau),
	\end{equation} where \(C\neq 0\) is a constant, \(m,m_1,m_2\) are integers, \(0\leq m_1, m_2 \leq m\), and \(F\in C^n_b(\R\times\R)\) for every \(n\in \mathbb N\).
	
	This can be proved inductively. We have the \(k = 0\) case immediately by setting \(C = 1\), \(m=m_1=m_2 = 0\), and \(F(x) = 1\). Now we assume that the claim holds for \(k\geq 0\). To get the form of the \((k+1)^\text{st}\) derivative, we can use linearity and look at the derivative of each term of the form \cref{term}. That is, the \((k+1)^\text{st}\) derivative is a sum of terms of the form 
	\begin{equation}\label{term2}
		\frac{\partial}{\partial x}\left[\frac C {\langle x +\tau\rangle_+^{2+m} \langle x -c\tau \rangle^{2+m}}\langle x + \tau \rangle_+^{m_1}\langle x-c\tau\rangle^{m_2} F(x,\tau)\right].
	\end{equation}
	Applying the product rule to \cref{term2} gives us 
	\begin{align*}
		\frac{\partial}{\partial x}\Bigg[&\frac C {\langle x +\tau\rangle_+^{2+m} \langle x -c\tau \rangle^{2+m}} \langle x + \tau \rangle_+^{m_1}\langle x-c\tau\rangle^{m_2} F(x,\tau)\Bigg] = \\
		&\qquad\quad \underbrace{\frac{\partial}{\partial x}\left[\frac C {\langle x +\tau\rangle_+^{2+m} \langle x -c\tau \rangle^{2+m}}\right]\langle x + \tau \rangle_+^{m_1}\langle x-c\tau\rangle^{m_2} F(x,\tau)}_{I} \\
		&\qquad+ \underbrace{\frac C {\langle x +\tau\rangle_+^{2+m} \langle x -c\tau \rangle^{2+m}}\frac{\partial}{\partial x}\left[\langle x + \tau \rangle_+^{m_1}\right]\langle x-c\tau\rangle^{m_2} F(x,\tau)}_{II} \\
		&\qquad+ \underbrace{\frac C {\langle x +\tau\rangle_+^{2+m} \langle x -c\tau \rangle^{2+m}}\langle x + \tau \rangle_+^{m_1}\frac{\partial}{\partial x}\left[\langle x-c\tau\rangle^{m_2}\right] F(x,\tau)}_{III} \\
		&\qquad+ \underbrace{\frac C {\langle x +\tau\rangle_+^{2+m} \langle x -c\tau \rangle^{2+m}}\langle x + \tau \rangle_+^{m_1}\langle x-c\tau\rangle^{m_2} \frac{\partial}{\partial x}\left[F(x,\tau)\right] }_{IV}.
	\end{align*}
	
	We now go term-by-term. For the first term, we have
	\begin{align*}
		I =& \frac {-(2+m)C} {\langle x +\tau\rangle_+^{2+(m+1)} \langle x -c\tau \rangle^{2+(m+1)}}\langle x + \tau \rangle_+^{m_1+1}\langle x+\tau\rangle^{m_2} \Big(\langle x-c\tau\rangle_+'F(x,\tau)\Big) \\
		&\quad-\frac {(2+m)C} {\langle x +\tau\rangle_+^{2+(m+1)} \langle x -c\tau \rangle^{2+(m+1)}}\langle x + \tau \rangle_+^{m_1}\langle x-c\tau\rangle^{m_2+1} \Big(\langle x-c\tau\rangle'F(x,\tau)\Big),
	\end{align*}
	where \(\langle \cdot \rangle'\) denotes the derivative of \(\langle \cdot \rangle\). It's clear that both of these are of the form in \cref{term}.
	
	Also, we have 
	\begin{align*}
		II = \frac {Cm_1} {\langle x +\tau\rangle_+^{2+m} \langle x -c\tau \rangle^{2+m}}\langle x + \tau \rangle_+^{m_1-1}\langle x-c\tau\rangle^{m_2} \Big( \langle x+\tau\rangle_+'F(x,\tau)\Big).
	\end{align*}
	The above is again of the form in \cref{term} (and a similar result holds for \(III\)). Finally,
	\begin{equation*}
		IV = \frac C {\langle x +\tau\rangle_+^{2+m} \langle x -c\tau \rangle^{2+m}}\langle x + \tau \rangle_+^{m_1}\langle x-c\tau\rangle^{m_2} \frac{\partial F}{\partial x}(x,\tau),
	\end{equation*}
	which of the form in \cref{term}.
	
	This shows that the \((k+1)^\text{st}\) derivative is a sum of terms of the form in \cref{term} and proves the claim.
	
	Now the proposition can be proved fairly straight-forwardly from the claim. The \(k^\text{th}\) derivative is a sum of terms of the form in \cref{term}, each of which can be bounded as
	\begin{align*}
		\left|\frac C {\langle x +\tau\rangle_+^{2+m} \langle x -c\tau \rangle^{2+m}}\langle x + \tau \rangle_+^{m_1}\langle x-c\tau\rangle^{m_2} F(x,\tau)\right| \\
		\leq C \| F \|_{C^0(\R\times \R)} \sup_{x\in\mathbb R} \frac 1 {\langle x+\tau\rangle_+^2 \langle x -c\tau \rangle^2}.
	\end{align*}
	The constant in \cref{Ck-bound} can be chosen to be the sum of the constants in the above inequality. Note that there is no \(\tau\) dependence since we are taking the supremum of \(F\) over all \(x\) and \(\tau\).
	
	The result in \cref{sup-integrable} follows from 
	\begin{equation*}
		\sup_{x\in\R} \frac 1 {\langle x+\tau\rangle_+^2 \langle x -c\tau \rangle^2} = \mathcal O (1/\tau^2)
	\end{equation*}
	as \(\tau\to \infty\).
\end{proof}

\begin{proof}[Proof of \cref{ell22-lemma}]
	Let \(E_n := \{k \in \Z \mid k \leq n\}\) so that the characteristic function \(\chi_{E_n}\) satisfies
	\begin{equation*}
		\chi_{E_n}(k) = \begin{cases}1, & k \leq n \\ 0, & k > 0\end{cases}.
	\end{equation*}
	Then applying the Cauchy-Schwarz inequality, we get that
	\begin{align*}
		\left| \sum_{k=-\infty}^n a_k \right| &= \left| \sum_{k=-\infty}^\infty \langle k\rangle ^2a_k \frac{\chi_{E_n}(k)}{\langle k \rangle^2} \right|  \\
		&\leq \| a\|_{\ell^2_2} \left( \sum_{k=-\infty}^\infty \frac{\chi_{E_n}(k)}{\langle k \rangle^4}\right)^{1/2} \\
		&= \| a\|_{\ell^2_2} \left( \sum_{k=-\infty}^n \frac{1}{\langle k \rangle^4}\right)^{1/2}.
	\end{align*}
	By comparing the final sum to the integral \(\int_{-\infty}^n 1/ \langle x\rangle^4 \, dx\), we have that there is a constant \(C>0\) independent of \(a\) such that
	\begin{equation*}
		\left| \sum_{k=-\infty}^n a_k \right|  \leq C \| a\|_{\ell^2_2} \times \frac 1 {\langle n\rangle^{3/2}}
	\end{equation*}
	for \(n\leq 0\). By noting that \(\sum_{k=-\infty}^n a_k = -\sum_{k=n+1}^\infty a_k\), an identical argument can be applied to get that 
	\begin{equation*}
		\left| \sum_{k=n}^\infty a_k \right|  \leq C \| a\|_{\ell^2_2} \times \frac 1 {\langle n\rangle ^{3/2}}
	\end{equation*}
	for \(n \geq 0\). Therefore,
	\begin{equation*}
		\|b \|_{\ell^2} \leq C \left( \sum_{n=-\infty}^\infty \frac{1}{\langle n\rangle^{3}}\right)^{1/2} \| a \|_{\ell^2_2}.
	\end{equation*}
\end{proof}

	\bibliography{LongTimeStability.bib}

\end{document}